\documentclass[a4paper,11pt,DIV=12, headings=standardclasses, headings=small]{scrartcl}

\usepackage{nag}
\usepackage[utf8]{inputenc}
\usepackage[english]{babel}

\usepackage{amscd} 

\usepackage{amsmath}
\usepackage{amsthm}
\usepackage{amsfonts}
\usepackage{amssymb}
\usepackage{mathrsfs} 
\usepackage{arydshln} 
\usepackage{url} 
\usepackage{graphicx}
\usepackage{tikz}
\usepackage{tikz-cd}
\usetikzlibrary{arrows,arrows.meta, shapes, shapes.geometric, shapes.misc, trees, graphs, matrix,  automata, backgrounds, calendar, positioning,external}
\tikzcdset{arrow style=tikz, diagrams={>=stealth}}

\usepackage{pxfonts}
\usepackage[T1]{fontenc} 

\usepackage[pdfpagelabels,hidelinks]{hyperref}
\usepackage[textsize=tiny,background color=white]{todonotes}
\usepackage[pdfpagelabels,hidelinks]{hyperref}
\usepackage[numbers]{natbib}

\DeclareMathOperator{\ev}{ev}

\DeclareMathOperator{\Iso}{Iso}

\DeclareMathOperator{\End}{End}
\DeclareMathOperator{\coEnd}{coEnd}
\DeclareMathOperator{\Id}{Id}
\newcommand{\id}{\text{id}}

\DeclareMathOperator{\Hom}{Hom}

\DeclareMathOperator{\Aut}{Aut}
\DeclareMathOperator{\colim}{colim}
\DeclareMathOperator{\GL}{GL}

\newcommand{\Z}{\mathbb{Z}}
\newcommand{\Q}{\mathbb{Q}}
\newcommand{\R}{\mathbb{R}}

\newcommand{\Ss}{\mathbb{S}}

\newcommand{\gL}{\mathfrak{g}}

\newcommand{\cC}{{\mathcal C}}

\newcommand{\TOP}{\mathsf{Top}}

\newcommand{\cdga}{{\mathsf{cdga}}}

\newcommand{\De}{\Delta}

\newcommand{\ra}{{\rightarrow}}

\newcommand{\lmt}{\longmapsto}
\newcommand{\hra}{\hookrightarrow}
\newcommand{\lra}{\longrightarrow}

\newcommand{\half}{{{\frac{1}{2}}}}

\DeclareMathOperator{\B}{B}

\DeclareMathOperator{\haut}{hAut}

\DeclareMathOperator{\Emb}{Emb}
\DeclareMathOperator{\Map}{Map}

\DeclareMathOperator{\fw}{fw}
\DeclareMathOperator{\fr}{fr}
\DeclareMathOperator{\Fr}{Fr}

\DeclareMathOperator{\op}{{op}}

\DeclareMathOperator{\Diff}{Diff}

\DeclareMathOperator{\Bij}{\mathsf{Bij}}

\DeclareMathOperator{\FM}{FM}
\newcommand{\hcoker}{/\!\!/}

\DeclareMathOperator{\TwGra}{Tw\,Gra}

\DeclareMathOperator{\Graphs}{Graphs}
\DeclareMathOperator{\GC}{GC}
\DeclareMathOperator{\fGC}{fGC}
\DeclareMathOperator{\osp}{\mathfrak{osp}}
\DeclareMathOperator{\MC}{MC}
\DeclareMathOperator{\Tw}{Tw }
\DeclareMathOperator{\Gra}{Gra}
\DeclareMathOperator{\sgn}{sgn}
\DeclareMathOperator{\pr}{pr}

\DeclareMathOperator{\Com}{\emph{Com}}
\DeclareMathOperator{\Lie}{\emph{Lie}}
\DeclareMathOperator{\hoLie}{\emph{hoLie}}

\numberwithin{equation}{section}
 
\makeatletter
\newtheorem*{rep@theorem}{\rep@title}
\newcommand{\newreptheorem}[2]{%
\newenvironment{rep#1}[1]{%
 \def\rep@title{#2 \ref{##1}}%
 \begin{rep@theorem}}%
 {\end{rep@theorem}}}
\makeatother

\newtheorem{thm}{Theorem}[section]
\newreptheorem{theorem}{Theorem}

\newtheorem{cor}[thm]{Corollary}
\newtheorem{lem}[thm]{Lemma}
\newtheorem{prop}[thm]{Proposition}

\newtheorem{thmL}{Theorem}

\theoremstyle{definition}
\newtheorem{defn}[thm]{Definition}
\newtheorem{rem}[thm]{Remark}

\theoremstyle{remark}

\makeatletter
\renewenvironment{proof}[1][\proofname] {\par\pushQED{\qed}\normalfont\topsep6\p@\@plus6\p@\relax\trivlist\item[\hskip\labelsep\bfseries#1\@addpunct{.}]\ignorespaces}{\popQED\endtrivlist\@endpefalse}
\makeatother

\title{\Large Characteristic classes of framed fibre bundles}
\author{\large \textsc{Nils Prigge}}
\date{}

\begin{document}

\maketitle

\begin{abstract}\noindent
	\textbf{Abstract.} We generalize Kontsevich's construction of characteristic classes of fibre bundles with homology sphere fibres and a trivialization of the vertical tangent bundle from \cite{Ko02} to framed fibre bundles with closed manifold fibres.
\end{abstract}

\section{Introduction}

In his seminal paper \cite{Ko02}, Kontsevich gave a new construction of characteristic classes of certain fibre bundles. More precisely, let $M^d$ be a smooth, odd dimensional homology sphere and $\pi:E\ra B$ a smooth fibre bundle with fibre $D_M:=M^d\setminus \text{int } D^d$. Furthermore, assume that the fibrewise boundary is trivialized and that we have a framing of the vertical tangent bundle $T_{\pi}E$ which agrees with a fixed framing of $TD_M|_{\partial D_M}$ under the trivialization of the fibrewise boundary. For such bundles, Kontsevich constructs a chain map $I:\GC_d\ra \Omega_{dR}(B)$ from a certain graph complex $\GC_d$.  Roughly, the graph complex is a cochain complex with basis given by isomorphism classes of connected graphs $\Gamma$ (with some orientation data) whose vertices are at least trivalent, the degree of a graph is $|\Gamma|=(d-1)\cdot E(\Gamma)-d\cdot V(\Gamma)$, and the differential is given by the sum over all edge contractions. The map $I:\GC_d\ra \Omega_{dR}(B)$ is defined via certain configuration space integrals and is independent (up to chain homotopy) of all choices and natural with respect to pullbacks, and thus describes characteristic classes of such bundles.

\smallskip
Kontsevich's construction is remarkable in many ways. For one, it has been shown by Watanabe that these characteristic classes are non-trivial \cite{Wa09I,Wa09II}, and Watanabe used them to show that $\pi_*(\Diff_{\partial}(D^d))_{\Q}\neq 0$ in certain degrees far outside the range that has been accessible with traditional tools. But maybe even more importantly, the conception of graph complexes has sparked development and connections in many seemingly unrelated areas in mathematics (see \cite{Wil15,CGP21,K97,BM14} to name but a few). The most relevant for this paper is the application in the rational homotopy theory of the little discs operad $E_d$ or equivalently the Fulton-MacPherson operad $\FM_d$ \cite{FTW17,Fr17I,Fr17II}, and the right $\FM_d$-modules associated to framed manifolds $M$ given by the Fulton-MacPherson compactification $\FM_M$ of the ordered configurations spaces \cite{CW23,Idr19,W23}. 

Albeit the many advances in the application of graph complexes in recent years, the original geometric construction of the characteristic classes of Kontsevich has remained somewhat mysterious. And although many results have been obtained featuring similar graph complexes, they do not arise in a geometric but rather homotopical way and the exact connection appears somewhat unclear. The main motivation in writing this paper was to understand how  Kontsevich's characteristic classes relate to other results where similar graph complexes appear and to get a better conceptual understanding of them. For example, it is not clear initially what role the various assumptions play and whether they are actually necessary. In fact, results on the rational homotopy theory of configuration space modules over the (framed) little disks operad suggests that there should be generalizations for all manifolds. Such a generalization for framed fibre bundles (see below) with closed fibres is the main result of this paper, and the construction sheds some light on the connection to the results about the rational homotopy theory of $E_d$-modules where similar graph complexes appear.

\medskip
We start by sketching the combinatorial definition of the graph complex that appears in our main theorem and which has been introduced in \cite{CW23}. Let $M^d$ be a closed, connected, smooth manifold of dimension $d$. Denote by $\GC_{H(M)}$ the graph complex spanned by (possibly infinite sums of) isomorphism classes of connected graphs $\Gamma$ where the vertices can have decorations in the reduced homology groups $\overline{H}_*(M;\R)$ together with some orientation data that we describe in detail later on. Note that we do not impose any valency conditions and that we allow multiple edges and loops.\footnote{The orientation data of a graph depends on the parity of $d$ and multiple edges can only occur when $d$ is odd whereas loops can only occur when $d$ is even.} A simple example of a graph in $\GC_{H(M)}$ without specifying the orientation data is visualized below
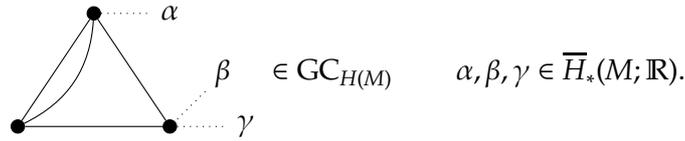
\begin{equation*}
\vcenter{\hbox{
\begin{tikzpicture}[vertex/.style={circle,fill=black,draw,minimum size=5pt,inner sep=0.2pt]},decoration/.style={}]
    \node[vertex] (1) at (0,0) {};
    \node[vertex] (2) at (-1, -1.5)  {};
    \node[vertex] (3) at (1, -1.5) {};
    \node[decoration] (4) at (1,0) {$\alpha$};
    \node[decoration] (5) at (1.7,-0.8) {$\beta$};
    \node[decoration] (6) at (2,-1.5) {$\gamma$};
    \draw (1) -- (2) -- (3) -- (1);
    
    \draw[dotted] (1) -- (4);
    \draw[dotted] (5) -- (3) -- (6);
    \path (1) edge[bend left] (2);
\end{tikzpicture}}}\in \GC_{H(M)} \qquad \alpha,\beta,\gamma\in \overline{H}_*(M;\R).
\end{equation*}
The differential has two contributions. First, we sum over all possible ways of expanding a single vertex into two vertices joined by an edge and sum over all ways of connecting the incoming edges to the two created vertices (this is the dual of the edge contraction differential above in $\GC_d$). Secondly, we sum over all possible ways of replacing two decorations by an edge weighted by the intersection pairing of the chosen decorations. Here, any vertex should be thought to be decorated by $1\in H_0(M;\R)$ so that it may be joined with a decoration of the fundamental class of $M$. The graph complex $\GC_{H(M)}$ can be further given the structure of a dg Lie algebra, where the Lie bracket of two graphs is given by the sum of all possible ways of replacing one decoration of each graph by an edge between them, again weighted by the intersection pairing of the two chosen decorations.

In \cite[Lem.\ 19]{CW23} they construct for closed, framed manifolds $M$ a cdga map
\begin{equation}\label{PartitionFunctionCW}
Z_M:\cC_{CE}^*(\GC_{H(M)})\lra \R,
\end{equation}
which they call the \emph{partition function}, using configuration space integrals and which encodes the real homotopy type of $\FM_M$ as right $\FM_d$-module by \cite[Thm 25]{CW23}. Here, $\cC^*_{CE}(-)$ denotes the Chevalley-Eilenberg cochain complex of a complete Lie algebra of finite filtration type\footnote{This notation is non-standard but if $L$ is a dg Lie algebra of finite type, it agrees with the usual definition of the Chevalley-Eilenberg cochain complex.} (see Section \ref{SectionFiltration}) and $\cC_{CE}^*(\GC_{H(M)})$ agrees with the full graph complex in \cite{CW23}.

\smallskip
Our first main result provides a generalization of this construction for a \emph{framed fibre bundle} $\pi:E\ra B$ with fibre $M$, i.e.\ a smooth submersion with a trivialization $\tau_E:T_{\pi}E\xrightarrow{\cong} E\times \R^d$ of the vertical tangent bundle, using a slightly larger version of the graph complex. We denote by $\osp^{<0}_{H(M)}$ the ortho-symplectic Lie algebra, i.e.\ linear maps that preserve the cup product pairing on the cohomology $H^*(M;\R)$, which acts on the decorations of graphs in $\GC_{H(M)}$ and we consider the semi-direct product $\osp^{<0}_{H(M)}\ltimes \GC_{H(M)}$. 
\begin{thmL}[Theorem \ref{PartFctZE}]
  Let $\pi:E\ra B$ be a framed fibre bundle of connected spaces with closed fibre $M$ and $\dim M>2$ so that $\pi_1(B)$ acts trivially on $H^*(M;\R)$. One can define a fibrewise partition function via configuration space integrals
  \[Z_E: \cC^*_{CE}(\osp^{<0}_{H(M)}\ltimes \GC_{H(M)})\lra \Omega_{dR}(B),\] which is an invariant of the framed fibre bundle that is natural with respect to pullbacks.
\end{thmL}
\begin{rem}
\leavevmode
\begin{itemize}
 \item[(i)] The content of the theorem lies entirely in the construction of the partition function $Z_E$. Unfortunately, we cannot say much about it other than that the construction is the same as Kontsevich's construction for homology disk bundles.
 \item[(ii)] The homology of Kontsevich's original graph complex is the object of current research and there are many known classes, but it is still largely unknown. Even less is known about the (Chevalley-Eilenberg co-) homology of the graph complex in our main theorem. But at least in principle this problem is amenable to combinatorial methods that potentially yield new characteristic classes.
\end{itemize}
\end{rem}
The information encoded in the partition function can be refined as follows. There is a one-to-one correspondence between cdga maps $\cC^*_{CE}(\GC_{H(M)})\ra \R$ and Maurer-Cartan elements in $\GC_{H(M)}$, and we denote by $z_M\in \GC_{H(M)}$ the Maurer-Cartan element corresponding to the partition function from \eqref{PartitionFunctionCW}. There is a version of $\GC_{H(M)}$ of graphs containing only vertices of valence $\geq 3$ (counting decorations) denoted by $\GC_{H(M)}^{\geq 3}$, and $z_M$ is equivalent to a Maurer-Cartan element $z_M^{\geq 3}\in \GC_{H(M)}^{\geq 3}$ by \cite[Prop.\ 46]{CW23} in an essentially unique way. The action of the ortho-symplectic Lie algebra preserves the valency condition and we define 
\begin{equation}\label{glM}
 \gL_M:=(\osp_{H(M)}^{<0}\ltimes \GC^{\geq 3}_{H(M)})^{z^{\geq 3}_M}\langle 0 \rangle,
\end{equation}
where for a dg Lie algebra $L$ we denote by $L\langle 0 \rangle\subset L$ the truncated dg Lie algebra of elements of non-positive degree (and cocycles in degree $0$).
\begin{thmL}[Theorem \ref{MainTheorem}]
 Let $\pi:E\ra B$ be a framed fibre bundle of connected spaces with closed fibre $M$ and $\dim M>2$ so that $\pi_1(B)$ acts trivially on $H^*(M;\R)$. Let $\gL_M$ be as above for $z^{\geq 3}_M\in \GC_{H(M)}^{\geq 3}$ describing the real homotopy type of the Fulton-MacPherson compactification of the fibres of $\pi$ as $\FM_d$-module. Then there is a map
 \begin{equation}\label{IMainTheorem}
  I:\cC_{CE}^*(\gL_M)\lra \Omega_{dR}(B)
 \end{equation}
 which is natural with respect to pullbacks and thus describes characteristic classes of framed $M$-bundles.
\end{thmL}
\begin{rem} The dg Lie algebra $\gL_M$ (defined over $\Q$) has appeared in \cite[Cor.\ 13.8]{W23} as a model for the classifying space of automorphisms of the rationalized configuration space module $\FM_M^{\Q}$, which was in fact the starting point for this project as we explain in Section \ref{Perspective}.
\end{rem}

\section*{Acknowledgements} 
I would like to heartily thank Thomas Willwacher for repeatedly explaining his results about algebraic models for configuration spaces as well as sharing with me his (then) unpublished result in \cite{W23} that was the starting point for this project. I also thank Alexander Berglund for comments on this paper. The author was supported by ERC grant 678156 (GRAPHCPX) as well as by the Knut and Alice Wallenberg foundation through grant no. 2019.0519.

\section{Preliminaries}
Throughout the paper we consider cohomology groups of a space $X$ with real coefficients that we denote by $H(X)$ unless stated otherwise. All algebraic objects are differential $\Z$-graded (or just dg) vector spaces over $\R$ and we use cohomological conventions, i.e.\,all differentials have degree $+1$. We denote the desuspension of a graded vector space $V$ by $V[n]$, i.e. $V[n]^k=V^{k+n}$.

\smallskip

We make limited use of the language of (co)operads and (co)modules because of the important conceptual role of the module structure of the Fulton-MacPherson compactifications of ordered configuration spaces in the construction of the characteristic classes, as well as the notational ease that it provides.

Denote by $\Bij$ the category of finite (possibly empty) sets with bijections as morphisms. We think of operads in some monoidal category $(\mathsf{C},\otimes,\mathbf{1})$ as indexed by finite (and possibly empty) sets, i.e.\,functors $O:\Bij^{\op}\ra \mathsf{C}$, together with structure maps that are either given by partial composition maps  
\[\circ_V:O(U/V)\otimes O(V)\lra O(U)\]
for finite sets $V\subset U$ that satisfy certain associativity and naturality conditions (see \cite{LV12,Fr17II} for details). Observe that $U/\emptyset =U\sqcup \{*\}$, so that if $O(\emptyset)=\mathbf{1}$ we obtain forgetful maps \[\pi_{V\subset U}:O(U)\ra O(V)\] by successive applications of the operadic structure map $\circ_{\emptyset}$ corresponding to $\emptyset \subset V$ and identifying iterated quotients of $V$ by the empty set with $U$.

Similarly, a right module $M$ over an operad $O$ is a functor $M:\Bij^{\op}\ra \mathsf{C}$ with structure maps 
\[\circ_V:M(U/V)\otimes O(V)\lra M(U)\]
that satisfy associativity and naturality conditions. Again, if $O(\emptyset)=\mathbf{1}$ there are forgetful maps 
\[\pi_{V\subset U}:M(U)\ra M(V)\]
for all $V\subset U$. We denote by $O(n)$ respectively $M(n)$ the value on the finite sets $\underline{n}:=\{1,\hdots,n\}$ which have a right action of the symmetric groups $S_n$, and we thus recover the classical definition of operads and modules. For cooperads and right comodules the structure maps simply go in the other direction subject to the analogous associativity and naturality conditions. 

In the following, the category $\mathsf{C}$ is either the category of topological spaces $\TOP$ or the category of differential graded commutative algebras $\cdga$.

\subsection{Compactifications of configuration spaces}\label{IntroFM}

The Fulton-MacPherson compactification of ordered configuration spaces of a smooth manifold $M$ plays a central role in this paper. Following \cite{Sin04}, it can be described in terms of global coordinates as the compactification of an embedding of the ordered configuration space in some ambient space. We denote the space of configurations labelled by a finite set $U$ by $\text{Inj}(U,M)$, and if $U=\underline{n}$ we denote the corresponding ordered configuration space by $C_n(M)$. Fixing an embedding $M\hra \R^N$, there is an embedding of $\text{Inj}(U,M)$ as follows
\begin{equation}\label{FMclosure}
\begin{split}
	e:\text{Inj}(U,M)&\lra M^U\times \prod_{C_2(U)}S^{N-1}\times\prod_{C_3(U)}[0,\infty],\\
	(x_u)_{u\in U}&\lmt \left((x_u)_{u\in A},\left(\frac{x_u-x_v}{||x_u-x_v||}\right)_{(u,v)\in C_2(U)},\left(\frac{||x_u-x_v||}{||x_u-x_w||}\right)_{(u,v,w)\in C_3(U)}\right)
\end{split}
\end{equation}
and we define the Fulton-MacPherson compactification of $U$-labelled configurations as
\begin{equation}
	\FM_M[U]:=\overline{\text{im}(e)}\subset  M^U\times \prod_{C_2(U)}S^{N-1}\times\prod_{C_3(U)}[0,\infty].
\end{equation}
If $M$ is a manifold without boundary then Sinha showed that $\FM_M[U]$ is a manifold with corners and if $M$ is compact then so is $\FM_M[U]$ (see \cite[Thm 4.4 and Prop.\,1.4]{Sin04}). A similar construction of $U$-labelled configurations in $\R^d$ (modulo translations and dilations) gives compact manifolds with corners that we denote by $\FM_d[U]$. These manifolds assemble to the \emph{Fulton-MacPherson operad} $\FM_d=\{\FM_d[n]\}_{n\geq 0}$ (see \cite[Ch.\,5]{LV14} for an excellent account).

If $M$ is a framed manifold, then the collection of compactified configuration spaces $\FM_M:=\{\FM_M[n]\}_{n\geq 0}$ forms a right module over $\FM_d$. The right module structure plays an important conceptual role in this paper and we briefly introduce some definitions and notation below.

\medskip
The Fulton-MacPherson operad carries a natural action of $\GL^+_d(\R)$ and the \emph{framed Fulton-MacPherson operad} is defined as the semi-direct product $\FM_d^{\fr}:=\FM_d\rtimes \GL^+_d(\R)$ (see \cite{SW03} for the definition of the semi-direct product). Similarly, we can consider configurations with tangential data. Denote by $\Fr^+(TM)$ the oriented frame bundle which is a principal $\GL^+_d(\R)$-bundle over $M$. Then the framed Fulton-MacPherson compactification $\FM_M^{\fr}[U]$ is defined as the pullback 
\begin{equation}\label{FramedPullback}
\begin{tikzcd}
\FM^{\fr}_M[U]\arrow{d}\arrow{r} & \Fr^+(TM)^U\arrow{d}\\
\FM_M[U] \arrow{r} & M^U,
\end{tikzcd}
\end{equation}
and the collection $\FM_M^{\fr}:=\{\FM_M^{\fr}[n]\}_{n\geq 0}$ is a right $\FM_d^{\fr}$-module. By restricting along the operad map $\FM_d\ra \FM_d^{\fr}$, the framed Fulton-MacPherson compactification is also a right $\FM_d$-module. A framing $\tau:TM\ra M\times \R^d$ determines a section $s_{\tau}:M\ra \Fr^+(TM)$ of the oriented frame bundle and by \eqref{FramedPullback} a section $s^U_{\tau}:\FM_M[U]\ra \FM_M^{\fr}[U]$. The composition
\[\FM_M[U/V]\times \FM_d[V]\xrightarrow{s_{\tau}\times \Id}\FM_M^{\fr}[U/V]\times \FM_d[V]\xrightarrow{\circ_V} \FM_M^{\fr}[U]\ra \FM_M[U]\]
is the (partial) structure map corresponding to $V\subset U$ that give $\FM_M$ the structure of a right $\FM_d$-module. We end by introducing the following notation for all $V\subset U$ 
\begin{equation}\label{ParticleProjection}
\pi_{V\subset U}:\FM_M[U]\lra \FM_M[V],
\end{equation}
for the map that forgets all (coordinates of) particles that are not in $V$ (usually, $U$ is clear from context and to simplify the notation we denote the forgetful map by $\pi_V$). If $M$ is framed, this agrees with the forgetful map of the $\FM_d$-module structure corresponding to $V\subset U$.

\subsection{Real models for configuration spaces}\label{GraphModels}
In this section, we recall the results from \cite{CW23}. It will be illuminating for subsequent applications to fibre bundles to start with the description of a general graph complex that can be thought of as the source of all possible configuration space integrals. It is elegantly described via the theory of operadic twisting \cite[App.\,I]{Wil15}, which we use to give precise, algebraic definitions. But we rely on the graphical interpretation to simplify the exposition when convenient.
\begin{defn}\label{GraV} For $d$ a natural number we define the following functors $\Bij^{\op}\ra \cdga$:
	\begin{itemize}
		\item[(i)] For a finite set $U$ denote by $\text{Gra}_d(U)$ the free graded commutative algebra on symbols $s^{uu'}$ for $u\neq u'$ of degree $d-1$ with trivial differential and modulo the relation $(-1)^ds^{u'u}=s^{uu'}$.
		\item[(ii)] Let $(V,\langle, \rangle)$ be a graded vector space $V$ with a splitting in degree zero $V=\R\oplus \overline{V}$ and with a non-degenerate, symmetric pairing $\langle,\rangle: V\otimes V\ra \R$ of degree $-d$. The \emph{diagonal class} is defined as
		\[\Delta_{\langle, \rangle_V}=\sum_i(-1)^{|v_i|}v_i\otimes v_i^{\#}\in V\otimes V,\]
		where $\{v_i\}$ is a basis of $V$ and $v_i^{\#}$ is the dual basis with respect to the pairing. This definition does not depend on the choice of basis.
		
		For a finite set $U$ we denote by $\Gra_{V,\langle \rangle}(U)$ the free graded commutative algebra on symbols $s^{uu'}$ for $u,u'\in U$ of degree $d-1$ and $\bigoplus_U \overline{V}$ modulo the relation $(-1)^ds^{u'u}=s^{uu'}$. The differential is determined by setting $d|_{\bigoplus_U\overline{V}}=0$ and 
		\begin{equation}\label{dsplit}
			d(s^{uu'})=\pi_u^*\otimes \pi_{u'}^*(\De_{V,\langle \rangle})\in V_u\otimes V_{u'}\subset S(\oplus_U\overline{V}),
		\end{equation}
		where $V_u^0=\R$ is identified with the unit in $S(\oplus_U\overline{V})$ and $\pi^*_u$ denotes the map induced by the inclusion $\overline{V}_u\ra \bigoplus_U\overline{V}$.
	\end{itemize}
The prototypical example for a vector space with non-degenerate pairing for us is the cohomology ring of a closed oriented manifold $M$ with the Poincar\'e pairing, and we denote the corresponding functor by $\Gra_{H(M)}$ in this case.
\end{defn}	
A monomial $\prod_{i\in I} s^{u_iu'_i}\in \text{Gra}_d(U)$ can be represented pictorially by a graph  $\Gamma$ with set of vertices $V(\Gamma)=U$ and set of edges $E(\Gamma)=\{(u_i,u'_i)\}_I$ with an order of $E(\Gamma)$ if $d$ is even or directed edges if $d$ odd. Similarly, monomials in $\Gra_{V,\langle\rangle}$ are represented by graphs with vertex set $U$ and decorations of vertices in $\overline{V}$. The differential splits the an edge and replaces it with decorations in the diagonal class. 
\begin{rem}\label{DevenDodd}
It follows by symmetry that for $d$ odd the graphs can have no loops/tadpoles since $s^{vv}=0$, whereas for $d$ even the graphs cannot have multiple edges.
\end{rem}

\begin{lem}[{\cite[Prop.\,12,13]{CW16}}]
	$\Gra_d$ is a cooperad in $\cdga$ and $\Gra_{V,\langle \rangle}$ is a right $\Gra_d$-comodule.
\end{lem}
It is straightforward to describe the $\circ_V$-compositions in this context. For $V\subset U$ we define 
\begin{align*}
	\circ_V:\Gra_d(U)&\lra \Gra_d(U/V)\otimes \Gra_d(V)
\end{align*}
on generators by 
\[ \circ_{V}(s^{uu'})= \begin{cases}
s^{\bar{u}\bar{u}'}\otimes 1 +1\otimes s^{uu'} & u,u'\in V\\
s^{\bar{u}\bar{u}'}\otimes 1 &  \text{otherwise}\end{cases}\]
where $\bar{u},\bar{u}'\in U/V$ denote the image under the projection $U\ra U/V$. Observe that for odd $d$ the structure map simplifies as there are no tadpoles.

\medskip
The operadic twisting of Willwacher associates to the right $\Gra_d$-comodule $\Gra_{V,\langle,\rangle}$ a right comodule $\TwGra_{V,\langle,\rangle}$ over a cooperad $\TwGra_d$ that is called the \emph{twisted graph complex}. The spaces of cooperations are defined as
\begin{equation}\label{TwGra}
\begin{split}
	\TwGra_d(U)&:=\bigoplus_{j\geq 0}\left(\Gra_d(U\sqcup \underline{j})\otimes \R[d]^{\otimes j}\right)_{S_j}\\
	\TwGra_{V,\langle\rangle}(U)&:=\bigoplus_{j\geq 0}\left(\Gra_{V,\langle\rangle}(U\sqcup \underline{j})\otimes \R[d]^{\otimes j}\right)_{S_j}.
\end{split}
\end{equation}
We can represent elements $\Gamma=[\gamma\otimes 1^n]\in \Tw \Gra_{V,\langle\rangle}(U)$ for  
\begin{equation}\label{TwGraElement}
\gamma=\prod_{k=1}^D\pi_{v_k}^*\alpha_{k} \cdot \prod_{l=1}^E s^{u_lu'_l}\in \Gra_{V,\langle \rangle}(U\sqcup \underline{n})\otimes \R[d]^{\otimes n}
\end{equation}
by a graph with vertices $V(\Gamma)=U\sqcup \underline{n}$, where $U$ is the set of external vertices and $n$ unlabelled internal vertices, set of edges $E(\Gamma)=\{(u_l,u'_l)\}_{l=1,\dots,E} \in U\sqcup \underline{n}$ and decorations $\alpha_1,\dots,\alpha_D\in \overline{V}$ at vertices $v_1,\dots,v_D \in V(\Gamma)$. The orientation of a graph is given by an ordering of the decorations and an ordering of the edges if $d$ is even, respectively an order of the internal vertices and directed edges if $d$ is odd. The multiplication in the twisted graph complex is defined in the appendix in \eqref{TwMmultiplication}. Pictorially, it is represented by gluing two graphs along the external vertices. We refer to \cite{Wil15} and \cite{CW23} for the details of the construction such as the cooperad and comodule structure, and we discuss the definition of the differential, which is the most relevant part for this paper, in the appendix. Pictorially, the differential on $\TwGra_{V,\langle \rangle}(\emptyset)$ has two contributions \[d=d_{\text{split}}+d_{\text{contr}},\]
where the first part is coming from the internal differential on $\Gra_{V,\langle \rangle}$ and corresponds to splitting an edge and decorating with the diagonal class, and the second corresponds to the sum over all (non-loop) edge contractions (see \eqref{dcont} for the exact expression).

\medskip
With these definitions in place, we can come to the heart of the construction -- the configuration space integral. For a closed, oriented manifold $M$ one can extend the global angular form of $S(TM)=\partial \FM_M[2]$ to a so-called propagator classes $\phi_{12}\in \Omega_{PA}(\FM_M[2])$ that satisfies $d\phi_{12}=\sum (-1)^{|\alpha_i|}\pi_1^*\overline{\alpha}_i\wedge\pi_2^* \overline{\alpha}^{\#}_i\in \Omega_{PA}^d(\FM_M[2])$ for representatives $\overline{\alpha}_i\in \Omega_{PA}(M)$ of a basis $\{\alpha_i\}$ of $ H(M)$.\footnote{For technical reasons the above construction uses the cdga of piecewise semi-algebraic forms $\Omega_{PA}(-)$ discussed in \cite{HLTV11} which is quasi-isomorphic to the de Rham complex. This technicality is not important in this work and we suggest to just think of differential forms.} We denote the pullback of $\phi_{12}$ along $\pi_{uu'}:\FM_M[U]\ra \FM_M[2]$ by $\phi_{uu'}$. Then, using the same notation as in \eqref{TwGraElement}, to each graph $\Gamma=[\gamma\otimes 1^n]\in \TwGra_{H(M)}(U)$ we associate a PA-form 
\[\omega(\gamma)=\bigwedge_{k=1}^D\pi_{v_k}^*\overline{\alpha}_{i_k}\wedge \bigwedge_{l=1}^E\phi_{u_lu'_l} \in \Omega_{PA}(\FM_M[V(\Gamma)])\]
and consider its fibre integral along the forgetful map $\pi_{V_{\text{ext}(\Gamma)}}:\FM_M[V(\Gamma)]\ra \FM_M[V_{ext}(\Gamma)]$. Observe that while $\omega(\gamma)$ depends on the representative $\gamma$, the fibre integral along $\pi_{V_{\text{ext}(\Gamma)}}$ does not. Also, in the case that $\Gamma$ contains a loop one has to adapt the propagator as we discuss in Section \ref{PropagatorSection}.

\begin{lem}[{\cite[Lem.\,18]{CW16}}]\label{CSIMap}
	For a closed, framed manifold $M$ there is a map of right cooperadic comodules 
	\begin{equation}\label{CSI}
	\begin{split}
	I_{\bullet}:\TwGra_{H(M)}&\lra \Omega_{PA}(\FM_M)\\
		\Gamma=[\gamma\otimes 1^n]&\lmt \int_{\FM_M[V(\Gamma)]\ra \FM_M[V_{ext}(\Gamma)]}\omega(\gamma)
	\end{split}
	\end{equation}
	that assigns to a graph the corresponding configuration space integral.
\end{lem}

We discuss the construction and properties of propagators of smooth fibre bundles as well as the definition of the above map in detail in subsequent sections and we defer more details of the construction there. 

\smallskip
For now, it is important to note that the map \eqref{CSI} serves as the basic ingredient for the definition of the combinatorial model for $\FM_M$ which captures the right $\FM_d$-module structure for a framed manifold $M$ in the following way. Campos and Willwacher define the full graph complex as the cdga of graphs without external vertices $\fGC_{H(M)}:=\TwGra_{H(M)}(\emptyset)$. Observe in particular that $\TwGra_{H(M)}(U)$ is an algebra over $\fGC_{H(M)}$ for every finite set $U$. We further denote the restriction of \eqref{CSI} to $\fGC_{H(M)}$ by
\begin{equation}\label{ZM}
	Z_M:\fGC_{H(M)}\lra \R,
\end{equation}
which is also referred to as the partition function due to the connection to theoretical physics. 
\begin{thm}[\cite{CW23}]
	For a closed, framed manifold $M$ the induced map 
	\begin{equation}\label{CSImap}
		\Graphs_M:=\TwGra_{H(M)}\otimes_{Z_M}\R\lra \Omega_{PA}(\FM_M)
	\end{equation}
	is a quasi-isomorphism of comodules.
\end{thm}

\begin{rem}\label{RemarkCamposWillwacher}\leavevmode
	\begin{itemize}
		\item[(i)] This reduces the problem to determining a combinatorial model of $\FM_M$ to computing the partition function $Z_M$, which involves computing a possibly infinite number of configuration space integrals and so is still quite difficult. In \cite{CW23} they further show that the information encoded in the partition function is the same as the (naive) real homotopy type of $M$ if $\dim M\geq 4 $ and $H^1(M)=0$. 
		\item[(ii)] In fact, Campos and Willwacher prove a version of the above statement for every closed manifold. One has to slightly modify the combinatorial models by removing tadpoles, but one can still define a map as in \eqref{CSI} to  obtain a quasi-isomorphism of symmetric sequences in $\cdga$ as in \eqref{CSImap}. An even more general version with rational coefficients has been obtained in \cite{W23} by completely algebraic tools. 
	\end{itemize}
\end{rem}

There is a convenient reformulation of the information encoded in the partition function $Z_M$ described in \cite[Sect.\,7]{CW16}. The full graph complex $\fGC_{H(M)}$ is isomorphic to a free graded commutative algebra on the subspace of connected vacuum graphs. As the differential can increase the number of connected components by at most one, this gives the subspace of connected graphs a dg Lie coalgebra structure.
\begin{defn}
The dg Lie algebra $\GC_{H(M)}$ is the dual of the dg Lie coalgebra $\fGC_{H(M)}$.
\end{defn}
Hence, as described in the introduction, elements in $\GC_{H(M)}$ are represented by infinite sums of graphs decorated by elements in  $\overline{H}_*(M)$. The Lie bracket $[\Gamma,\Gamma']$ of two graphs, which is the dual of the quadratic part of the differential of $\fGC_{H(M)}$, is given by summing over all possible ways of selecting a decoration in $\Gamma $ and $\Gamma'$ and joining them in an edge with a factor determined by the intersection pairing. The differential, which is the dual of the linear part of the differential of $\fGC_{H(M)}$, is given by vertex splitting and joining decorations of a graph. In the next section, we discuss the identification of cdga maps $\fGC_{H(M)}\ra \R$ with Maurer-Cartan elements in $\GC_{H(M)}$.

\subsection{Complete dg Lie algebras}\label{SectionFiltration}

Another important feature of the dg Lie algebra $\GC_{H(M)}$ is that it admits a filtration which allows to pass between the dg Lie algebra and its pre-dual dg Lie coalgebra. 
\begin{defn}
	A dg Lie algebra is \emph{complete} if there is a descending filtration by dg Lie ideals $L=F^1L\supset F^2L\supset\hdots $ so that $[F^iL,F^jL]\subset F^{i+j}L$ and $L\cong\text{lim}_r \,L/F^rL$. A complete dg Lie algebra is of \emph{finite filtration type} if the quotients $L/F^kL$ are finite dimensional for all $k$.
\end{defn}
We can associate to a complete dg Lie algebra $(L,F^{\bullet}L)$ of finite filtration type a pre-dual dg Lie coalgebra  \[{}^*L:=\colim\, (L/F^kL)^{\vee},\] where the dg coalgebra structure is encoded in the colimit of the Chevalley-Eilenberg cochain complexes of $L/F^kL$ that we denote by $\cC_{CE}^*(L)$. This is non-standard terminology for the Chevalley-Eilenberg complex but is well suited for complete dg Lie algebras of finite filtration type.

For $\GC_{H(M)}$ we define filtrations given by the ideals $F^k\GC_{H(M)}$ of graphs satisfying
\[d\cdot \#\text{edges}+\#\text{decorations}+\#\text{vertices}\geq k\]
where we count decorations with respect to their (homological) grading. This filtration is compatible with the bracket and differential and gives $\GC_{H(M)}$ the structure of a complete dg Lie algebra of finite filtration type, and by definition the Chevalley-Eilenberg cochain complex $\cC_{CE}^*(\GC_{H(M)})$ agrees with the full graph complex $\fGC_{H(M)}$ of Campos and Willwacher.

\smallskip
Given a cdga $A$ and a complete dg Lie algebra $L$ of finite filtration type, the completed tensor product $A\hat{\otimes} L:= \underset{\leftarrow}{\text{lim }}A\otimes L/F^kL$ is a complete dg Lie algebra and there is an identification 
\begin{equation}\label{MCIdentification}
	\MC(A\hat{\otimes} L)\cong \Hom_{\cdga}(\cC_{CE}^*(L),A).
\end{equation}
For any complete dg Lie algebra $\mathfrak{g}$ one can define its Maurer-Cartan space 
\[ \MC_{\bullet}(\mathfrak{g}):=\MC(A_{PL}(\Delta^{\bullet})\hat{\otimes}\mathfrak{g}),\]
where $A_{PL}(\Delta^{k})$ denotes the cdga of polynomial differential form on the $k$-simplex. The set of path-components $\pi_0(\MC_{\bullet}(\mathfrak{g}))$ can be identified with the gauge equivalence classes, and later we use versions of the Goldman-Milson that show that under mild assumptions the homotopy type of the Maurer-Cartan space does not change under quasi-isomorphisms of dg Lie algebras.

\section{Models for fibrations}\label{models}
The construction of the map \eqref{CSI} relied on a choice of representatives $H_{dR}(M)\ra \Omega_{dR}(M)$ to fibre integrate the decorations of graphs in $\TwGra_{H(M)}$. The goal of this section is to discuss the appropriate generalisation for smooth fibre bundles $\pi:E\ra B$ with fibre $M$. We will see that this requires some conditions on the action of $\pi_1(B)$ on $H(M)$ and that the correct model for $\Omega_{dR}(E)$ is given by a quasi-isomorphism 
\begin{equation}\label{Mmodel}
\phi:(\Omega_{dR}(B)\otimes H_{dR}(M),D)\overset{\simeq}{\lra} \Omega_{dR}(E)
\end{equation}
of $\Omega_{dR}(B)$-modules that is suitably compatible with the Poincar\'e duality pairing.\footnote{The main result of this section, Proposition \ref{ospFC}, is proved independently in forthcoming work of Berglund \cite{Be25} by more sophisticated techniques.}

\medskip
We use the main theorem from \cite{Ha83} regarding the minimal models of fibrations in the construction of the quasi-isomorphism in \eqref{Mmodel}. 
\begin{thm}[{\cite[Thm\,20.3]{Ha83}}]\label{Halperin}
	Let $\pi:E\ra B$ be a fibration of path-connected spaces and path-connected fibre $X$. Let  
	\begin{equation*}
		\begin{tikzcd}
		A_{PL}(B)\arrow{r}{A_{PL}(\pi)} & A_{PL}(E) \arrow{r} & A_{PL}(X)\\
		A_{PL}(B)\arrow[equals]{u} \arrow{r}& R\arrow{u}  \arrow{r} & T\arrow{u}{\alpha}.
		\end{tikzcd}
	\end{equation*}
	be a minimal relative Sullivan model. If the following conditions hold
	\begin{itemize}
		\item[(i)] $H(X;\Q)$ is a nilpotent $\pi_1(B)$-module,
		\item[(ii)] Either $H(X;\Q)$ or $H(B;\Q)$ have finite type,
	\end{itemize}
	then $\alpha:T\ra A_{PL}(X)$ is a minimal model for $X$.
\end{thm}
The following lemma is a direct consequence of this theorem and is also mentioned in \cite{HaTa90} but not explicitly stated.
\begin{lem}\label{FC}
	Let $\pi:E\ra B$ be a smooth fibration of connected spaces with fibre $X$ of finite $\R$-type and nilpotent action of $\pi_1(B)$ on $H(X)$. Then there exists a differential $\Omega_{dR}(B)$-module $(\Omega_{dR}(B)\otimes H(X),D)$ with a quasi-isomorphism 
	\begin{equation}\label{FCmodel}
	\phi:(\Omega_{dR}(B)\otimes H(X),D)\xrightarrow{\,\simeq\,} \Omega_{dR}(E)
	\end{equation}
	of $\Omega_{dR}(B)$-modules.
\end{lem}

\begin{proof}	
	Consider relative Sullivan models of $A_{PL}(B)\ra A_{PL}(E)$ and $\Omega_{dR}(B)\ra \Omega_{dR}(E)$ that we denote by $(A_{PL}(B)\otimes \Lambda V',D')$ and $(\Omega_{dR}(B)\otimes \Lambda V,D)$ respectively. There is a commutative diagram of quasi-isomorphisms 
	\begin{equation*}
		\begin{tikzcd}
		A_{PL}(E)\arrow{r}{\simeq} & A_{\infty}(E) & \Omega_{dR}(E) \arrow[swap]{l}{\simeq}\\
		A_{PL}(B)\arrow{u}{A_{PL}(\pi)} \arrow{r}{\simeq} & A_{\infty}(B) \arrow{u}{A_{\infty}(\pi)}& \Omega_{dR}(B) \arrow[swap]{l}{\simeq}\arrow{u}{\Omega_{dR}(\pi)}
		\end{tikzcd}
	\end{equation*}
	where $A_{\infty}(X)$ denotes the cdga of smooth differential forms on the singular complex of a space $X$. By change of basis they induce relative Sullivan models of the middle map which are isomorphic by \cite[Thm\,6.2]{Ha83}. In particular, the Sullivan fibre of $(\Omega_{dR}(B)\otimes \Lambda V,D)$ is isomorphic to the Sullivan fibre of $(A_{PL}(B)\otimes \Lambda V',D')$, which is a minimal model for $X$ by Theorem \ref{Halperin}. 
	
	Hence, we can find a chain homotopy equivalence of $(H(X),0)$ and $(\Lambda V,d)$ which induces an equivalence of $\Omega_{dR}(B)$-modules
	\begin{equation}\label{HEdata}
		\begin{tikzcd}
		\Omega_{dR}(B)\otimes (H(X),0) \arrow[shift left=0.3ex]{r} & \Omega_{dR}(B)\otimes (\Lambda V,d)\arrow[loop right]{}{h} \arrow[shift left=0.3ex]{l}
		\end{tikzcd}
	\end{equation}
	 Consider the increasing filtrations on both sides given by $F^{p}=\Omega^{\geq b- p}_{dR}(B)\otimes -$. Since the above chain homotopy equivalence is induced by one between $(H(X),0)$ and $(\Lambda V,d)$, the filtration is preserved by all maps in \eqref{HEdata}. By \cite[Cor.\,1.13]{Ha83} one can assume that $D(1\otimes v)-d(1\otimes v)\in \Omega^+_{dR}(B)\otimes \Lambda V$ so that $D$ is obtained by a perturbation of the differential on $\Omega_{dR}(B)\otimes (\Lambda V,d)$ that decreases  the filtration. Hence, we can apply the basic perturbation lemma (see for example \cite[Sect.\,2.4]{GL89}) to find a differential on $\Omega_{dR}(B)\otimes H(X)$ with a quasi-isomorphism of $\Omega_{dR}(B)$-modules to $(\Omega_{dR}(B)\otimes \Lambda V,D)\simeq \Omega_{dR}(E)$.
\end{proof}
In the following, we denote a model as in Lemma \ref{FC} by \[\Omega(E):=(\Omega_{dR}(B)\otimes H(X),D)\overset{\simeq}{\lra}\Omega_{dR}(E).\]
By inspection of the perturbation lemma, the model is natural with respect to restriction and it follows from \cite[Sect.\ 20.6]{Ha83} that it is compatible with fibrewise product.
\begin{lem}\label{ModelNaturality}
	The above model is natural with respect to restrictions, i.e.\,for any open subset $U\subset B$ the induced map 
	\[\pi^*\otimes \phi|_{\pi^{-1}(U)}:\Omega_{dR}(U)\otimes_{\Omega_{dR}(B)}(\Omega_{dR}(B)\otimes H(X),D)\lra \Omega_{dR}(\pi^{-1}(U))\]
	is a quasi-isomorphism of $\Omega_{dR}(U)$-modules.
\end{lem}
\begin{lem}\label{Cproduct}
	The model is compatible with fibrewise products, i.e. \[\phi\otimes\phi:\Omega(E)\otimes_{\Omega_{dR}(B)}\Omega(E)\ra \Omega_{dR}(E)\otimes_{\Omega_{dR}(B)}\Omega_{dR}(E)\ra \Omega_{dR}(E\times_BE)\] is a quasi-isomorphism.
\end{lem}
It follows from Lemma \ref{ModelNaturality} that the local coefficient system $\mathcal{H}(X;\R)$ of a fibration $\pi:E\ra B$ with fibre $X$ and nilpotent action of $\pi_1(B)$ is trivial as a continuous vector bundle (this can also be seen from an induction over the nilpotency class of $\mathcal{H}(X)$). For a model $\Omega(E)$ from Proposition \ref{FC} we can decompose the differential as $D=\sum_{i\geq 1} D^i$ with $D^i\in \Omega^i_{dR}(B)\otimes \End^{1-i}(H(X))$. The first part of the differential $D^1$ defines a flat connection which is the same as a twisting matrix in the sense of Sullivan \cite[Sect.\,1]{Sul77}, and can be used to compute the local cohomology $\mathcal{H}(X)$.

\begin{lem}\label{Torelli}
	If $\pi:E\ra B$ is a fibration with fibre $X$ and trivial action of $\pi_1(B)$ on $H(X)$, then we can find a model $\Omega(E)$ with $D=\sum_{i>1}D^i$.
\end{lem}
\begin{proof}
	Pick a model $\Omega(E)$ from Proposition \ref{FC} and identify $\Omega_{dR}(B)\otimes H(X)$ with differential forms on $B$ with values in the sections of the trivial (graded) vector bundle $H_{\fw}(E)=\coprod_{b\in B}H(\pi^{-1}(b))\ra B$. The condition $D^2=0$ implies that $D^1$ defines a flat connection on this trivial vector bundle. Since the holonomy action is trivial by assumption, we can find a frame of global flat sections with respect to $D^1$. 	
	The differential with respect to this $\Omega_{dR}(B)$-module basis of global flat sections satisfies $D=\sum_{i>1}D^i$.
\end{proof}
From now on we assume that $\pi:E\ra B$ is a smooth submersion of connected spaces with fibre $M^d$ a closed oriented manifold. We say that the bundle is \emph{oriented} if the action of $\pi_1(B)$ on $H_{dR}^d(M)$ is trivial. In this case, we can define \emph{fibre integration} $\int_M: \Omega_{dR}(E)\ra \Omega_{dR}(B)$ which is a map of $\Omega_{dR}(B)$-modules (see Appendix \ref{FibreIntegration} for the definition and conventions). It induces a $\Omega_{dR}(B)$-bilinear pairing on a model $\Omega(E)$ from Lemma \ref{FC}
\begin{equation}\label{pairing}
\begin{split}
\langle\,,\,\rangle :\Omega(E)\otimes_{\Omega_{dR}(B)}\Omega(E)\xrightarrow{\phi\otimes\phi} \Omega_{dR}(E)\otimes_{\Omega_{dR}(B)}\Omega_{dR}(E)\lra\Omega_{dR}(B)[ d]\\
\alpha\otimes \beta\lmt \int_{\pi}\phi(\alpha)\wedge \phi(\beta)
\end{split}
\end{equation}
which satisfies by construction 
\[d\langle x,y\rangle =\langle D(x),y\rangle +(-1)^{|x|}\langle x,D(y)\rangle.\] 
Given a basis $\{x_i\}$ of $H_{dR}(M)$, we denote $1\otimes x_i\in \Omega(E)$ by $ x_i$ as well and the form $\phi(x_i)\in \Omega_{dR}(E)$ by $\overline{x}_i$. Below we also require that the fundamental group of the base acts trivially on $H(M)$, but we expect that the results in this section can be generalized for nilpotent action of $\pi_1(B)$ on $H(M)$ and further using the tools developed in \cite{BZ22}.
\begin{lem}\label{PDDiagonal}
	Let $\pi:E\ra B$ be a smooth submersion with closed, oriented fibre $M^d$ so that $\pi_1(B)$ acts trivially on $H(M)$. Let $\Omega(E)$ be a model from Lemma \ref{FC} and $\{x_i\}$ a basis of $H(M)$, then there exists a dual basis under the pairing $x_i^{\#}\in \Omega(E)$ satisfying $\langle x_i,x_j^{\#}\rangle=\delta_{i,j}$ and the Poincar\'e dual of the diagonal $\De:E\ra E\times _BE$ is represented by
	\begin{equation}\label{PDDiagonalRep}
		\De_!(1):=\sum_{i,j}(-1)^{|x_i|}x_i\otimes x_i^{\#}\in \Omega(E)\otimes_{\Omega_{dR}(B)}\Omega(E).
	\end{equation}
\end{lem}

\begin{proof}
The adjoint	of the pairing \eqref{pairing} determines a map $\Omega(E)\ra \Hom_{\Omega_{dR}(B)}(\Omega(E),\Omega_{dR}(B)[d])$ of $\Omega_{dR}(B)$-modules. By a spectral sequence argument it is a quasi-isomorphism (see for example \cite[Prop.\,4.1]{Pri19}). Since both sides are minimal semi-free $\Omega_{dR}(B)$-modules, this is in fact an isomorphism and so for a basis $\{x_i\}$ of $H(M)$ there is a dual basis $x_i^{\#}\in \Omega(E)$ satisfying $\langle x_i,x_j^{\#}\rangle=\delta_{i,j}$. 
	 
It remains to show that \eqref{PDDiagonalRep} is a representative of the Poincar\'e dual of the diagonal, i.e.\ that it is closed (which is a straightforward computation) and that
	 \begin{equation}\label{PoincareDualDiagonal}
	 \int_{E\times_BE\ra B}\Phi \wedge \Delta_!(1)= \int_{E\ra B}\De^*(\Phi)
	 \end{equation}
	 for all $[\Phi]\in H(E\times _BE)$. Since any class in $H_{dR}^*(E\times_BE)$ can, by Lemma \ref{Cproduct}, be represented by a form $\sum \omega_{ij}\wedge \pi_1^*\overline{x}_i\wedge \pi_2^*\overline{x}_j$ for some $\omega_{ij}\in \Omega_{dR}(B)$, equation \eqref{PoincareDualDiagonal} holds because of the following identity for all $i,j$
	 \begin{align*}
	 &\int_{E\times _BE\ra B}\pi_1^*\overline{x}_i \wedge\pi_2^* \overline{x}_j\wedge \left(\sum_k(-1)^{|x_k|}\pi_1^*\overline{x}_k\wedge \pi_2^*\overline{x}_k^{\#}\right)
	 \\ =&\sum_k(-1)^{|x_k|+|x_k||x_j|}(-1)^{d(|x_j|+|x_k^{\#}|-d)}\int_{E\ra B}\overline{x}_i\wedge\overline{x}_k\int_{E\ra B}\overline{x}_j\wedge \overline{x}^{\#}_k   
	 \\ =&\int_{E\ra B}\overline{x}_i\wedge \overline{x}_j,
	 \end{align*}
	 where we use a fibrewise version of Fubini's theorem (see Proposition \ref{FibrewiseFubini}).
\end{proof}
For an oriented fibre bundle $\pi:E\ra B$ with fibre a closed manifold $M$, the Poincar\'e dual of the diagonal agrees with the Euler class of the vertical tangent bundle \cite{HLLRW17,RW16} and hence we can conclude the following.
\begin{cor}
	With the assumptions as above, the Euler class of the vertical tangent bundle is represented by $e(T_{\pi}E)=\sum_{i}(-1)^{|x_i|}\overline{x}_i\wedge\overline{x}_i^{\#}\in \Omega_{dR}^d(E)$.
\end{cor}
Finally, we show that we can find a model $\Omega(E)$ so that the pairing \eqref{pairing} coincides with  the pairing induced by Poincar\'e duality for $H(M)$, i.e.\,there exists a model so that 
\[\langle 1\otimes x,1\otimes y\rangle =\langle x\cup y, [M]\rangle \]
for all $x,y\in H(M)$. The argument is similar to finding an orthogonal or symplectic basis in vector spaces with a non-degenerate bilinear form, and holds more generally for differential $R$-modules for some cdga $R$ with a non-degenerate pairing.

\smallskip
Let $R$ be a positively graded, unital cdga with $H^0(R)=\R$ with augmentation $\epsilon:R\ra \R$, and $M=(R\otimes V,D)$ a semifree $R$-module. A $R$-bilinear pairing on $M$ of degree $-d$ is a map of dg $R$-modules
\[\langle ,\rangle : M\otimes_R M\lra R[d].\]
It is symmetric if $\langle m,m'\rangle =(-1)^{|m|\cdot |m'|}\langle m',m\rangle$ for all $m,m'\in M$, and non-degenerate if the adjoint is an isomorphism. A pairing induces a pairing on the vector space $V\cong \R\otimes_RM$ denoted by $\langle ,\rangle _V$ which is non-degenerate (resp.\ symmetric) if $\langle,\rangle $ is non-degenerate (resp.\ symmetric). If we further assume that $D(V)\subset R^{\geq 2}\otimes V$, then for all elements $v,w\in V$ with $|v|+|w|=d$ 
\[ d_R\langle 1\otimes v,1\otimes w\rangle =\langle D(1\otimes v),1\otimes w\rangle +(-1)^{|v|}\langle 1\otimes v,D(1\otimes w)\rangle=0\]
for degree reasons, and since $H^0(R)=\R$ 
\begin{align}\label{DPairing}
\langle 1\otimes v,1\otimes w\rangle = \langle v,w\rangle _V .
\end{align}
Below we usually abuse notation and denote $1\otimes v\in M $ by $v$, so that the above identity just states that for elements $v,w$ with degree $|v|+|w|=d$ the pairing agrees with the induced pairing on $V$.
\begin{lem}\label{PairingTechnical}
Let $R$ and $M$ be as above with $D(V)\subset R^{\geq 2}\otimes V$ and $\langle ,\rangle:M\otimes_RM\ra R$ a non-degenerate, symmetric pairing of degree $-d$. Then there exists an $R$-module isomorphism $f:M\ra M$ so that $f\otimes \R:M\otimes_R\R\ra M\otimes_R\R$ is the identity and the diagram below commutes
 \begin{equation*}
  \begin{tikzcd}[ampersand replacement=\&]
   M\otimes_RM\arrow{r}{\id\otimes \langle ,\rangle_V }\arrow{d}{f\otimes f} \& R\arrow[equal]{d}\\
   M\otimes_RM\arrow{r}{\langle ,\rangle } \& R   
  \end{tikzcd}
 \end{equation*}
where $\langle ,\rangle_V$ is the induced non-degenerate pairing of degree $-d$ on $V\cong M\otimes_R\R$.
\end{lem}
\begin{proof}
We identify $V\cong \R\otimes_RM$ and denote by $p_1,\dots,p_m$ a basis of $V$ in degrees $\leq \lfloor \frac{d-1}{2}\rfloor$, by $q_1,\hdots,q_m$ the corresponding dual basis of $V$ in degrees $\geq \lceil \frac{d+1}{2}\rceil$, and if $d$ is even denote by $x_1,\dots,x_n$ a basis of $V$ in degree $d/2$. Then 
\begin{align*}
 \langle p_i,p_j\rangle_V&=0, & \langle p_i,q_j\rangle_V&=\delta_{i,j}, & \langle p_i,x_j\rangle_V&=0,\\
 \langle q_i,q_j\rangle_V&=0 & \langle q_i,x_j\rangle_V&=0,
\end{align*}
and the restriction of $\langle,\rangle_V $ to $\mathrm{span}(x_1,\hdots,x_m)\subset V$ is non-degenerate and symmetric if $d\equiv 0\ (\mathrm{mod}\ 4)$ and anti-symmetric if $d\equiv2\ (\mathrm{mod}\,4)$.

We prove the statement by induction on $m$. If $m=0$ the pairing on $M$ agrees with the pairing induced on $V$ for degree reasons so there is nothing to prove. If $m>0$, we order the basis so that $|p_1|\leq |p_2|\leq \dots |p_m|$ and $|q_1|\geq |q_2|\geq \dots |q_m|$. By abuse of notation we denote by $p_i,q_j\in M$ the elements $1\otimes p_i,1\otimes q_j$ and define 
\begin{align*}
 P_1&=p_1 & Q_1=q_1-\half \langle q_1,q_1\rangle \cdot p_1
\end{align*}
and
\begin{align*}
 P_i&=p_i-\langle p_i,q_1\rangle\cdot p_1 \qquad i=2,\dots m,\\
 Q_i&=q_i-\langle q_i,q_1\rangle\cdot p_1 \qquad i=2,\dots m,\\
 X_j&=x_j-\langle x_j,q_1\rangle\cdot p_1 \qquad j=1,\dots n.
\end{align*}
Then clearly $P_i,Q_i,X_j$ define a new basis of $M$ so that the change of basis isomorphism $f$ satisfies $f\otimes_R\R=\Id_V$. We claim that the non-degenerate pairing splits over the submodules $M_1=R\{P_1,Q_1\}$ and
\[M_2=R\{P_2,\dots,P_m,Q_2,\dots,Q_m,X_1,\dots X_m\} .\]
For example, denoting $\langle q_i,q_1\rangle=r_i\in R$ to simplify notation, we compute
\begin{align*}
\langle Q_i,Q_1\rangle &=\langle q_i-r_ip_1,q_1-\half r_1p_1\rangle\\
                       &=\langle q_i,q_1\rangle -r_i\langle p_1,q_1\rangle -(-1)^{|q_i|\cdot |r_1|}\half r_1\langle q_i,p_1\rangle+(-1)^{|p_1|\cdot |r_1|}\half r_ir_1 \langle p_1,p_1\rangle.
\end{align*}
Observe that $\langle p_1,p_1\rangle=0$ for degree reasons and that $|\langle q_i,p_1\rangle|\leq 0$ with equality only if $|q_i|=|q_1|$. But then $\langle q_i,p_1\rangle\overset{\eqref{DPairing}}{=}\langle q_i,p_1\rangle_V=0$ for $i=2,\dots m$, so that we conclude that 
\begin{align*}
\langle Q_i,Q_1\rangle &=\underset{=r_i}{\langle q_i,q_1}\rangle -r_i\underset{=1}{\langle p_1,q_1\rangle} \\
&=0.                        
\end{align*}
Similar computations show that
\begin{align*}
	\langle P_1,P_i\rangle &=0,   & \langle P_1,Q_i\rangle &=0, & \langle P_1,X_j\rangle =0\\
	\langle Q_1,P_i\rangle &=0,   & \langle Q_1,Q_i\rangle&=0, & \langle Q_1,X_j\rangle =0
\end{align*}
 for $i=2,\dots m$ and $j=1,\dots ,n$, which shows that the pairing splits into a direct sum of non-degenerate pairings over $M_1\oplus M_2= M$. Over $M_1$ we see directly that 
\begin{align*}
 \langle P_1,P_1\rangle&=0 & \langle P_1,Q_1\rangle &=1.
\end{align*}
If $|q_1|$ is odd, then $\langle q_1,q_1\rangle =0$ by symmetry and $\langle Q_1,Q_1\rangle =0$. If $|q_1|$ is even then 
\begin{align*}
 \langle Q_1,Q_1\rangle &=\langle q_1,q_1\rangle -\half \langle q_1,q_1\rangle \langle p_1,q_1\rangle -\half (-1)^{(2|q_1|-d)\cdot |q_1|}\langle q_1,q_1\rangle\langle q_1,p_1\rangle\\
 &=\half \langle q_1,q_1 \rangle -\half (-1)^{|q_1|\cdot (d-|q_1|)}\langle q_1,q_1\rangle\\
 &=0.
\end{align*}
Hence, on $M_1$ the pairing agrees with the pairing on $\langle ,\rangle _V$ restricted to $\mathrm{span}(p_1,q_1)$, and on $M_2$ we can find a basis so that the pairing agrees with the pairing induced by $\langle,\rangle_V$ restricted to $\mathrm{span}\{p_i,q_i,x_j\}_{i=2,\dots ,m}^{j=1,\dots ,n}$ using the induction hypothesis.
\end{proof}
Combining Lemma \ref{Torelli} and Lemma \ref{DPairing} we obtain a $\Omega_{dR}(B)$-module $\Omega(E)$ that models the total space and satisfies the following properties.
\begin{prop}\label{ospFC}
	Let $\pi:E\ra B$ be a submersion with fibre a closed, oriented manifold $M$ and so that $\pi_1(B)$ acts trivially on $H(M)$. Then there exists a model $\phi:\Omega(E)=(\Omega_{dR}(B)\otimes H_{dR}(M),D)\xrightarrow{\,\simeq\,} \Omega_{dR}(E)$ with $D=\sum_{i>1}D^i$ and so that the following diagram commutes
	\begin{equation}\label{Mpairing}
		\begin{tikzcd}
		\Omega(E)\otimes_{\Omega_{dR}(B)} \Omega(E)\arrow{r}{\langle\,,\,\rangle_M} \arrow{d}{\phi\otimes\phi}& \Omega_{dR}(B)\arrow[equals]{d}\\
		\Omega_{dR}(E)\otimes_{\Omega_{dR}(B)}\Omega_{dR}(E)\arrow{r}{\int_{\pi}} & \Omega_{dR}(B),
		\end{tikzcd}
	\end{equation}
	where $\langle\,,\,\rangle_M$ denotes the pairing induced from the non-degenerate pairing on $H(M;\Q)$.
\end{prop}

The results of this section can be stated in purely geometric terms. Namely, for a smooth submersion $\pi:E\ra B$ with fibre $M$, the differential in the model from Lemma \ref{FC} defines a flat connection on the graded vector bundle $\coprod_{b\in B}H_{dR}(\pi^{-1}(b))\ra B$ such that the associated complex computes the cohomology of $E$. If in addition the fibre is closed then Proposition \ref{ospFC} implies that the connection can be refined to take values in the following Lie algebra.
\begin{defn}
	The \emph{ortho-symplectic Lie algebra} $\mathfrak{osp}'_{H(M)}$ is the graded Lie algebra of linear maps $f:H^*(M)\ra H^{*+|f|}(M)$ that satisfy $f(1)=0$ and $\langle f(x),y\rangle +(-1)^{|f|\cdot |x|}\langle x,f(y)\rangle =0$. We denote by $\osp_{H(M)}^{<0}$ the Lie subalgebra of elements of negative degree.
\end{defn}

\begin{prop}\label{summary}
	Let $\pi:E\ra B$ be a submersion with fibre a closed, oriented manifold $M$ and so that $\pi_1(B)$ acts trivially on $H(M)$ and $\Omega(E)=(\Omega_{dR}(B)\otimes H(M),D)$. Then $D$ defines a Maurer-Cartan element $F\in \MC(\Omega_{dR}(B) \otimes\mathfrak{osp}_{ H(M)}^{<0})$ and there is a pushout square 
	\begin{equation}\label{pushout}
	\begin{tikzcd}
	\cC_{CE}^*(\osp_{H(M)}^{<0};H(M))\arrow{r} & (\Omega_{dR}(B)\otimes H(M),D)\\
	\cC_{CE}^*(\osp_{H(M)}^{<0})\arrow{u}\arrow{r} & \Omega_{dR}(B)\arrow{u}.
	\end{tikzcd}
	\end{equation}
	given by evaluating the Maurer-Cartan element on Chevalley-Eilenberg cochains.
\end{prop}
\begin{rem}
 One can further show that Maurer-Cartan element $F\in \MC(\Omega_{dR}(B) \otimes\mathfrak{osp}_{ H(M)}^{<0})$ is unique up to gauge equivalence, which was pointed out to the author by Alexander Berglund. We prove a weaker statement later which is sufficient to show that the fibrewise configuration space integral we define in Section \ref{SectionGCC} does not depend on the choice of $F$.
\end{rem}

\begin{proof}
Let $\mathfrak{gl}_{H(M)}'$ denote the graded Lie algebra of endomorphisms of $H(M)$. Then the restriction of $D$ to $1\otimes H_{dR}(M)$ defines an element in $\Omega_{dR}(B)\otimes \mathfrak{gl}_{H(M)}'$, i.e.\ an endomorphism valued differential form of total degree $1$, that we denote by $F$ and the condition $D^2=0$ is equivalent to $F$ being a Maurer-Cartan element. If $D=\sum_{i>1}D^i$ then the $F$ is contained in $\Omega_{dR}(B)\otimes\mathfrak{gl}_{H(M)}^{<0}$. And if the fibre is closed and we take a model from Proposition \ref{ospFC}, then for any $x,y \in H_{dR}(M)$ the pairing $\langle x,y\rangle$ agrees by assumption with the Poincar\'e pairing and in particular is constant. Hence, 
	\begin{equation}\label{local}
	0=d\langle x,y\rangle=\langle F(x),y\rangle +(-1)^{|x|}\langle x,F(y)\rangle.
	\end{equation}
Writing $F=\sum _i\omega_i\otimes f_i$ for linearly independent differential forms $\omega_i$ and endomorphisms $f_i:H_{dR}(M)\ra H_{dR}(M)$, then \eqref{local} implies that $f_i\in \osp_{H(M)}^{<0}$.

Finally, given a dg Lie algebra $L$, a cdga $A$ and a Maurer-Cartan element $\tau\in \MC(A\otimes L)$, then for any $L$-module $M$ the tensor product $A\otimes M$ is a $A\otimes L$-module and there is a well-defined map
\[\cC_{CE}^*(L;M)\lra (A\otimes M)^{\tau}\]
given by evaluating Chevalley-Eilenberg cochains on $\tau$. Moreover, if $M$ comes with a map from the trivial $L$-module, then both $A$ and $\cC_{CE}^*(L;M)$ are modules over $\cC_{CE}^*(L)$ and the evaluation map above defines an isomorphism 
\[\cC_{CE}^*(L;M)\otimes_{\cC_{CE}^*(L)}A\xrightarrow{\,\cong\,} (A\otimes M)^{\tau}. \]
As $(\Omega_{dR}(B)\otimes H(M))^{F}$ agrees with $\Omega(E)$ by construction, we obtain the pushout square \eqref{pushout}.
\end{proof}

\section{Propagators for bundles}\label{PropagatorSection}
In this section, we construct the propagator classes that are needed in the construction of the configuration space integral map \eqref{CSI}. But first we need to discuss the compactification of the fibrewise configuration spaces of smooth fibre bundles $\pi:E\ra B$ with fibre $M$, i.e.\ the fibre bundles for $n\geq 0$ with fibre $C_n(M)$ given by
\[C_n^{\fw}(E):=\bigg\{(e_1,\hdots,e_n)\in E^n\,\bigg|\ \begin{array}{c c}
                                               \pi(e_i)=\pi(e_j) & \forall i,j\\
                                               e_i\neq e_j & \forall i\neq j 
                                              \end{array}
\bigg\}\lra B.\]
It has been proven by Sinha \cite[Cor.\,4.9]{Sin04} that $\Diff(M)$ acts on the Fulton-MacPherson compactification and therefore we can define for all $n\geq 0$ fibrewise compactified configuration spaces as 
\begin{equation}\label{fwCompactifications}
\FM_E^{\fw}[n]:=P\times_{\Diff(M)}\FM_M[n],
\end{equation}
where $P\ra B$ is the principal $\Diff^+(M)$-bundle associated to the orientable $M$-bundle $\pi:E\ra B$. Moreover, if $\pi:E\ra B$ is a submersion then so is the projection $\pi:\FM_E^{\fw}[n]\ra B$.

\begin{lem}\label{FibrewiseModule}
 Let $\pi:E\ra B$ be a framed fibre bundle with closed fibre $M$. The collection $\FM_E^{\fw}:=\{\FM_E^{\fw}[n]\}_{n\geq 0}$ is a right $\FM_d$-module. 
\end{lem}
\begin{rem}
 For smooth fibre bundles without framing one can consider the fibrewise framed configurations spaces defined as 
 \[\FM_{E}^{\fr,\fw}[n]:=P\times_{\Diff^+(M)}\FM_M^{\fr}[n]\]
 which assemble into a right module over $\FM_d^{\fr}$.
\end{rem}
\begin{proof}
One can extract the right $\FM_d$-module structure for the framed configuration spaces $\FM_M^{\fr}[n]$ from \cite{Sin04}. In particular, it follows from \cite[Cor.\ 4.8]{Sin04} that the elementary insertions
\[\circ_V:\FM_M^{\fr}[U/V]\times \FM_d[V]\lra \FM_M^{\fr}[U]\]
are $\Diff^+(M)$-equivariant, where $\Diff^+(M)$ acts trivially on $ \FM_d[V]$. Hence, they induce a map of associated bundles
\[\FM_E^{\fr,\fw}[U/V]\times \FM_d[V] \cong P\times_{\Diff^+(M)}(\FM_M^{\fr}[U/V]\times \FM_d[V]) \xrightarrow{\circ_V} \FM_E^{\fw}[U]\]
that we denote abusively by $\circ_V$ and which gives $\FM_E^{\fr,\fw}$ the structure of a right $\FM_d$-module. A framing of the vertical tangent bundle determines a section
$s_{\tau}^U:\FM_E^{\fw}[U]\ra \FM_E^{\fr,\fw}[U]$ by the same argument as in Section \ref{IntroFM}. The composition
\[ \FM_E^{\fw}[U/V]\times \FM_d[V]\xrightarrow{s^{U/V}_{\tau}\times \Id} \FM_E^{\fr,\fw}[U/V]\times \FM_d[V] \xrightarrow{\circ_V} \FM_E^{\fr,\fw}[V]\lra \FM_E^{\fw}[U] \]
restricts to the elementary insertion of the $\FM_d$-module $\FM_{\pi^{-1}(b)}$ for each fibre, and hence make $\FM_E^{\fw}$ into a right $\FM_d$-module.
\end{proof}

The propagator is a differential form $\phi_{12}\in\Omega_{dR}(\FM_E^{\fw}[2])$ that is essentially constructed as follows: The (fibrewise) boundary $\partial \FM_E^{\fw}[2]\cong P\times_{\Diff(M)}\partial\FM_M[2]$ is an $S^{d-1}$-bundle over $E$ that can be identified with $S(T_{\pi}E)$, and the propagator is an extension of a global angular form of $S(T_{\pi}E)$ to the interior of $\FM_E^{\fw}[2]$ via a bump function. If $\pi_1(B)$ acts nilpotently on $H(M)$, we can further improve the propagator to be compatible with a $\Omega_{dR}(B)$-module model for $\Omega_{dR}(E)$ from Lemma \ref{FC} (or Lemma \ref{ospFC} if $\pi_1(B)$ acts trivially).
\begin{lem}\label{Prop}
	Let $\pi:E\ra B^b$ be a smooth submersion of closed manifolds with fibre $M^d$ satisfying the assumptions of Lemma \ref{FC} and let $\phi:(\Omega_{dR}(B)\otimes H_{dR}(M),D)\ra \Omega_{dR}(E)$ be a differential $\Omega_{dR}(B)$-module model. Denote by $\{x_i\}$ a basis of $H_{dR}(M)$, then there exists a propagator form $\phi_{12}\in \Omega_{dR}^{d-1}(\FM_E^{\fw}[2])$ that satisfies:
	\begin{itemize}
		\item[(i)] $(-1)^d\phi_{12}|_{\partial \FM_E^{\fw}[2]}\in \Omega_{dR}^{d-1}(\partial \FM_E^{\fw}[2])$ is a global angular form;
		\item[(ii)] $\phi_{12}$ is $(-1)^d$-symmetric under the permutation action;
		\item[(iii)] The propagator is a primitive of $\De_!(1)$, i.e. 
		\[d\phi_{12}=\sum (-1)^{|x_i|}\,\pi_1^*(\overline{x}_i)\wedge\pi_2^*(\overline{x}_i^{\#})\] where $\pi_i:\FM_E^{\fw}[2]\ra \FM_E^{\fw}[1]=E$ denotes the projection to the $i$-th point.
		\item[(iv)] Moreover, if $\pi:E\ra B$ is a framed fibre bundle with framing $\tau:T_{\pi}E\ra E\times \R^d$, the propagator can be chosen so that the restriction to the boundary $\partial\FM_E^{\fw}[2]$ under the induced identification $E\times S^{d-1}\cong \partial\FM_E^{\fw}[2]$ is given by 
		\begin{equation}\label{PropFr}
            \phi_{12}|_{\partial \FM_E^{\fw}[2]}=(-1)^d\emph{vol}_{S^{d-1}}+\eta\in \Omega_{dR}(E)\otimes \Omega_{dR}(S^{d-1}).
        \end{equation}
        The form $\eta\in \Omega^{d-1}_{dR}(E)$ satisfies $d\eta=\sum_i(-1)^{|x_i|}\bar{x}_i \wedge \bar{x}_i^{\#}$ by (iii) and vanishes if $d$ is odd by (ii).
	\end{itemize}
\end{lem}

\begin{proof}
	The proof is virtually the same as \cite[Prop.\,7]{CW16} with minor modifications that we give here for completeness. Let $(-1)^d\psi\in \Omega^{d-1}_{dR}(\partial \FM_E^{\fw}[2])$ be a global angular form of $S(T_{\pi}E)$, i.e.\,it satisfies $\int_{S(T_{\pi}E)\ra E}\psi=(-1)^d$ and the exterior derivative is  $d\psi=\pi^*(e(T_{\pi}E)) $ for some representative of the Euler class. By symmetrizing we can assume $\psi$ to be $(-1)^d$-symmetric and we extend it $(-1)^d$-symmetrically to $\FM_E^{\fw}[2]$ via a bump function $\rho$ on a collar of the boundary $\partial \FM_E^{\fw}[2]\subset \FM^{\fw}_E[2]$. Since $d(\rho\psi)|_{\partial\FM_E^{\fw}[2]}$ is basic, it is the pullback of a form on $E\times _BE$ that we denote by abuse of notation by $d(\rho \psi)\in \Omega_{dR}^d(E\times_BE)$ as well even though it needn't be a coboundary.
    
    Then $d(\rho \psi)\in \Omega_{dR}^d(E\times_BE)$ is the Poincar\'e dual of the diagonal $\De:E\ra E\times_BE$, i.e.\,for all $[\omega]\in H^{d+b}_{dR}(E\times_BE)$ we have $\int_{E\times_B E}\omega\wedge d(\rho \psi)=\int_E\omega|_{\De}$. To see this, we use Stokes' theorem and the fact that integration doesn't change on sets of zero measure, i.e.\ $\int_{E\times_BE}\alpha = \int_{\FM_E^{\fw}[2]}\pi_{12}^*\alpha$ for all $\alpha \in \Omega_{dR}(E\times_BE)$ and $\pi_{12}:\FM_E^{\fw}[2]\ra E\times_BE$ denotes the projection map. Then 
	\begin{align*}
	\int_{E\times_BE}\omega\wedge  d(\rho \psi)&=(-1)^{d+b} \int_{\FM^{\fw}_E[2]}d(\pi_{12}^*\omega\wedge \rho\psi)\\&=(-1)^{d+b}\int_{\partial \FM^{\fw}_E[2]}\pi_{12}^*\omega|_{\De}\wedge \psi\\& = (-1)^{d+b}(-1)^b\int_E\omega|_{\De}\int_{S(T_{\pi}E)\ra E}\psi \\&=\int_{E}\omega|_{\De}
	\end{align*}
	as claimed. Observe that the sign $(-1)^b$ in the third step is due to the fact that the orientation of $\partial \FM_E^{\fw}[2]$ induced by outward pointing vector field differs by $(-1)^b$ to the fibre bundle orientation of $\partial\FM_E^{\fw}[2]=S(T_{\pi}E)$ over $E$. It follows from Lemma \ref{PDDiagonal} that the representative of $\De_!(1)$ given by \eqref{PDDiagonalRep} and $d(\rho \psi)$ are cohomologous. Let $\beta\in \Omega_{dR}^{d-1}(E\times_BE)$ be a primitive of the difference, which can be symmetrized as both $d(\rho\psi)$ and our representative of $\De_!(1)$ are $(-1)^d$-symmetric. Then we can set $\phi_{12}:=(\pi_1\times\pi_2)^*\beta+\rho\psi$ which satisfies (i)-(iii). 
	Finally, for a framed fibre bundle we can choose as global angular form $\mathrm{vol}_{S^{d-1}}\in \Omega_{dR}(E)\otimes \Omega_{dR}(S^{d-1})$ under the identification $\partial \FM_E^{\fw}[2]\cong E\times S^{d-1}$. Then the above argument produces a propagator that satisfies \eqref{PropFr}.
\end{proof}

Observe that we can change the propagator class by adding the pullback of any cohomology class $H^{d-1}(E\times_BE)$ without affecting any of the properties (i)-(iv). Moreover, we can give further conditions on propagators so that they become unique in the following sense.
\begin{prop}\label{PropUsed}
	Let $\pi:E\ra B$ be a framed fibre bundle and $\Omega(E)$ a $\Omega_{dR}(B)$-module model, then there is a compatible propagator $\phi_{12}\in \Omega_{dR}(\FM_E^{\fw}[2])$ from Lemma \ref{Prop} that satisfies
	\begin{equation}\label{NormalizationEq}
	\int_{\pi_1:\FM_E^{\fw}[2]\ra E} \phi_{12}\wedge \pi_2^*\overline{\alpha}=0
	\end{equation}
	for all $\alpha \in \Omega(E)$. Furthermore, any two such propagators are cohomologous.
\end{prop}
\begin{proof}
	Such a propagator has been constructed for $B=*$ in \cite[Lem.\,3]{Cam10} and we merely generalize their construction to non-trivial base spaces. The idea is to simply change $\phi_{12}$ by pulling back differential forms $\pi_1^*\left(\int_{\pi_1:\FM_E^{\fw}[2]\ra M}\phi_{12}\wedge \pi_2^*\overline{x}^{\#}_i\right)\wedge \pi_2^*\overline{x}_i\in \Omega_{dR}(E\times_BE)$ which cancel the contribution in \eqref{NormalizationEq} for $\alpha=x_i^{\#}$. At the same time we have to preserve properties (i)-(iii), which can be implemented in the following systematic way. Consider the map
	\begin{align*}
		\Id\times \pi_2&:E\times _B\FM_E^{\fw}[\{2,3\}]\ra E\times_BE,
	\end{align*}
	where the domain is the fibrewise compactification of triples $(p_1,p_2,p_3)\in \pi^{-1}(B)$ with $p_2\neq p_3$.  It is a smooth submersion with fibre $M\setminus D^d$ and thus we can consider the following fibre integral
	\[\int_{\Id\times \pi_2}\phi_{23} \wedge(\Id\times \pi_3)^*\De_!(1),\]
	where $\De_!(1)=\sum_i(-1)^{|x_i|}\pi_1^*\overline{x}_i\wedge \pi_2^*\overline{x}_i^{\#}\in \Omega^d_{dR}(E\times_BE)$ denotes the diagonal class, $\Id\times \pi_3:E\times_B\FM_E^{\fw}[2]\ra E\times_BE$ is the map that picks out the first and third particle and $\phi_{23}\in \Omega_{dR}^{d-1}(\FM_E^{\fw}[\{2,3\}])$ is the propagator. This expression can be written as 
	\begin{align*}
		\int_{\Id\times \pi_2}\phi_{23} \wedge (\Id\times \pi_3)^*\De_!(1)=\sum_i\int_{\Id\times p_2} (-1)^{|x_i|}\phi_{23} \wedge \pi_1^*\overline{x}_i\wedge \pi_3^*\overline{x}_i^{\#}=\int_3\phi_{23}\wedge \De_{13},
	\end{align*}
	where the last equality should be thought of as symbolic notation for the fibre integral along $\pi_{12}:\FM_E^{\fw}[3]\ra \FM_E^{\fw}[2]$. In general, this is not defined as $\pi_{12}$ is not a submersion. However, as the form $\phi_{23}\wedge \De_{13}$ is pulled back from the (singular) configuration space $E\times_B\FM_E^{\fw}[\{2,3\}]$ 
	\begin{equation}\label{SingularConfiguration}
		\begin{tikzcd}
		\FM_E^{\fw}[3]\arrow{d}{\pi_{12}} \arrow{r}{\pi_{1}\times \pi_{23}} & E\times_B\FM_E^{\fw}[\{2,3\}]\arrow{d}{\Id\times \pi_2}\\
		\FM_E^{\fw}[2] \arrow{r}{\pi_1\times \pi_2} & E\times_BE
		\end{tikzcd}
	\end{equation}
	the fibre integral is given by naturality by $(\pi_1\times \pi_2)^*\int_{\Id\times \pi_2}\phi_{23}\wedge  (\Id\times \pi_3)^*\De_!(1)$ and thus in particular is well-defined. 
	
	By analogous arguments we can define the following form
	\[\lambda_{12}:= \int_3 \phi_{13} \wedge \De_{23}+ \int_3\phi_{23}\wedge \De_{13}-(-1)^d \int_{3,4}\phi_{34}\wedge \De_{13}\wedge\De_{24}\in \Omega_{dR}^{d-1}(\FM_E^{\fw}[2])\]
	which is pulled back from $ \Omega_{dR}^{d-1}(E\times_BE)$. A straight forward yet lengthy calculation shows that $\lambda_{12}$ is closed, $(-1)^d$-symmetric and satisfies 
	\[ \int_{2}\lambda_{12}\wedge p_2^*\overline{\alpha}=\int_{2}\phi_{12}\wedge p_2^*\overline{\alpha}\]
	for all $\alpha\in H_{dR}(M)$. It follows that $\phi_{12}-\lambda_{12}$ is a propagator satisfying properties (i)-(iv).
	
	\medskip
	It remains to show that two propagators $\phi_{12},\phi_{12}'$ satisfying (i)-(iv) and \eqref{NormalizationEq} are cohomologous. Their difference $\phi_{12}-\phi_{12}'\in \Omega_{dR}^{d-1}(\FM_E^{\fw}[2])$ is closed and hence defines a cohomology class. Since the restriction of $\phi_{12}$ and $\phi_{12}'$ satisfies \eqref{PropFr} for a fixed volume form of $S^{d-1}$, their difference is a cohomology class pulled back from $E\times_BE$. By Lemma \ref{Cproduct} the difference is cohomologous to $\sum_{i,j}\omega_{i,j}\pi^*_1\overline{x}_i\wedge \pi_2^*\overline{x}_j$ where $\omega_{i,j}\in \Omega_{dR}(B)$. Since $\int_2(\phi_{12}-\phi_{12}')\pi_2^*\alpha=0$ for all $\alpha\in \Omega(E)$, it follows from choosing $\alpha=\overline{x}_j^{\#}$ that $\sum_i\omega_{i,j}\overline{x}_i=0$ for all $j$. But this implies that $\sum_{i,j}\omega_{i,j}\pi^*_1\overline{x}_i\wedge \pi_2^*\overline{x}_j=\sum_j\pi^*_1\left(\sum_i\omega_{i,j}\overline{x}_i\right)\wedge \pi_2^*\overline{x}_j=0$.
\end{proof}

\section{Graph characteristic classes}\label{SectionGCC}
In this section, we associate an invariant to framed fibre bundles and trivial holonomy action via configuration space integrals that generalizes Kontsevich's construction. The first step in the proof is the same as in \cite{CW23}, namely we construct a map as in Lemma \ref{CSIMap} from the twisted graph complex $\TwGra_{H(M)}$ to forms on fibrewise configuration spaces via configuration space integrals. But as we have seen in Section \ref{models}, we need to include the action of $\osp_{H(M)}^{<0}$ in order to describe the differential of the image of the decorations in $\Omega_{dR}(E)$, i.e.\ the domain of the configuration space integral map has to be $\cC_{CE}^*(\osp_{H(M)}^{<0};\TwGra_{H(M)})$ where $\osp_{H(M)}^{<0}$ acts naturally on the decorations of graphs. But we only define the partition function 
\begin{equation}\label{PartitionFunction}
	Z_E:\cC_{CE}^*(\mathfrak{osp}_{H(M)}^{<0};\fGC_{H(M)})\lra \Omega_{dR}(B),
\end{equation}	
i.e.\ the restriction to vacuum graphs $\TwGra_{H(M)}(\emptyset)=\fGC_{H(M)}$, because the forgetful map $\pi:\FM_E^{\fw}[U]\ra \FM_E^{\fw}[V]$ for $V\subset U$ is not a submersion unless $|V|\leq 1$ and hence fibre integration of differential forms is not well-defined. The construction depends on the following two choices:
\begin{itemize}
	\item[(i)] a model $\Omega(E)=(\Omega_{dR}(B)\otimes H_{dR}(M),D)\xrightarrow{\phi} \Omega_{dR}(E)$ from Proposition \ref{summary};
	\item[(ii)] a propagator $\phi_{12}\in \Omega_{dR}^{d-1}(\FM_E^{\fw}[2])$ from Proposition \ref{PropUsed}.
\end{itemize}
Identifying 
\[\cC_{CE}^*(\mathfrak{osp}_{H(M)}^{<0};\fGC_{H(M)})\cong \cC_{CE}^*(\mathfrak{osp}_{H(M)}^{<0})\otimes \fGC_{H(M)} \]
as graded commutative algebras, the restriction of $Z_E$ to $\cC_{CE}^*(\osp_{H(M)}^{<0})$ is induced by evaluating the Maurer-Cartan element $F\in \MC(\osp_{H(M)}^{<0}\otimes \Omega_{dR}(B))$ that encodes the model from (i) as discussed in Proposition \ref{summary}. 

For a graph $\Gamma=[\gamma\otimes 1^n]\in \fGC_{H(M)}$ the partition function is given by the corresponding configuration space integral, which we spell out below. The set of edges decomposes as $E_1(\Gamma)\sqcup E_2(\Gamma)$, where $E_1(\Gamma)$ contains all tadpoles if $d$ is even and all multiple edges if $d$ is odd. Following the notation in Appendix \ref{AlgebraicGraph}, $\gamma$ is given by 
\begin{equation*}
	\gamma:=\prod_{k=1}^{D}\pi_{i_k}^*\alpha_k\prod_{l_1=1}^{E_1}s^{m^1_{l_1},n^1_{l_1}} \prod_{l_2=1}^{E_2}s^{m^2_{l_2},n^2_{l_2}}  \in \Gra_{H(M)}(|V(\Gamma)|).
\end{equation*}
I.e.\ we pick a labeling $V(\Gamma)\cong\{1,\hdots,|V(\Gamma)|\}$ and denote by $\alpha_1,\hdots,\alpha_D\in \overline{H}_{dR}(M)$ the decorations of $\Gamma$ at vertices $1\leq i_1,\hdots,i_D\leq |V(\Gamma)|$. We enumerate the edges for $i=1,2$ by $E_i(\Gamma)=\{(m^i_{l_i}\}_{l_i=1,\hdots,|E_i(\Gamma)|}$ where $m^i_{l}\leq n^i_{l}$ denote the adjacent vertices of an edge. Observe, that if $d$ is even then an edge in $E_1(\Gamma)$ is a tadpole and hence $m^1_{l}=n^1_{l}$ for all $1\leq l\leq |E_1(\Gamma)|$.

We associate to $\gamma$ the differential form $\omega(\gamma)\in \Omega_{dR}(\FM_E^{\fw}[|V(\Gamma)|])$ given by a product of propagators for each edge (between distinct vertices), the form $\eta\in \Omega_{dR}(E)$ from \eqref{PropFr} for each tadpole and for each decoration the corresponding differential form in $\Omega_{dR}(E)$ 
\begin{equation}\label{omegaGamma}
\begin{split}
\omega(\gamma)&= \begin{cases}
\bigwedge_{k=1}^D \pi_{i_k}^*\overline{\alpha}_k\wedge\bigwedge_{l_1=1}^{|E_1(\Gamma)|}\pi^*_{m^1_{l_1}}\eta\wedge \bigwedge_{l_2=1}^{|E_2(\Gamma)|}\pi^*_{m^2_{l_2},n^2_{l_2}}\phi_{12} & \text{if } d \text{ even}\\
\bigwedge_{k=1}^D \pi_{i_k}^*\overline{\alpha}_k\wedge\bigwedge_{l_1=1}^{|E_1(\Gamma)|}\pi^*_{m^1_{l_1},n^1_{l_1}}\phi_{12}\wedge \bigwedge_{l_2=1}^{|E_2(\Gamma)|}\pi^*_{m^2_{l_2},n^2_{l_2}}\phi_{12} & \text{if } d \text{ odd}
\end{cases}\\
&=\omega^{D,1}(\gamma)\wedge \bigwedge_{l_2=1}^{|E_2(\Gamma)|}\pi^*_{m^2_{l_2},n^2_{l_2}}\phi_{12},
\end{split} 
\end{equation}
where $\pi_{v,v'}:\FM_E^{\fw}[V(\Gamma)]\ra\FM_E^{\fw}[2]$ and $\pi_v:\FM_E^{\fw}[V(\Gamma)]\ra E$ denote the projections defined in \eqref{ParticleProjection}, and in the second equation we have simplified the notation by collecting the terms corresponding to decorations and tadpoles if $d$ is even or multiple edges if $d$ is odd into the form $\omega^{D,1}(\Gamma)$. Later on, we sometimes denote the pullback $\pi_{v,v'}^*\phi_{12}$ by $\phi_{vv'}$.

\smallskip
The partition function then associates to $\Gamma$ the following configuration space integral
\begin{equation}
Z_E(\Gamma):=(-1)^{{|V(\Gamma)|\choose 2}\cdot d}\int_{\FM_E^{\fw}[V(\Gamma)]\ra B}\omega(\gamma)\in \Omega_{dR}(B),
\end{equation}
where the orientation of the fibre is determined by the orientation of $M$ and the labeling of the vertices, and the orientation data of $\Gamma$ (as explained in Appendix \ref{AlgebraicGraph}) is so that this expression is well-defined.
\begin{rem}
 Our definition of the partition function contains an extra sign, which appears to be different than the conventions in \cite{CW23,Idr19}. The reason we introduce this sign is so that $Z_E$ is compatible with the multiplication in $\Tw \Gra_{H(M)}$ (as given in \eqref{TwMmultiplication}) as well as the fibrewise version of Fubini's theorem based on our definition of fibre integration (see Lemma \ref{FibrewiseFubini}). I have not been able to extract all definitions and conventions that were used in \cite{CW23,Idr19}, but the end result shouldn't depend on which conventions are used. 
\end{rem}
\begin{thm}\label{PartFctZE}
 Let $\pi:E\ra B$ be a framed, smooth submersion of connected spaces with closed fibre and $\dim M>2$ and suppose that $\pi_1(B)$ acts trivially on $H(M)$. Then 
 \begin{equation}
  Z_E:\cC_{CE}(\osp_{H(M)}^{<0};\TwGra_{H(M)}(0))\lra \Omega_{dR}(B)
 \end{equation}
is a well-defined map of cdga's and independent of the choice of model or propagator. Moreover, it is natural with respect to pullbacks.
\end{thm}

\begin{rem}\label{RemarkFamilyCSI}\leavevmode
	\begin{itemize}
	\item[(i)]	By considering a (framed) manifold as a framed bundle over a point, the partition function \eqref{PartitionFunction} recovers the construction from \cite{CW23} as the flat $\osp^{<0}_{H(M)}$-connection is trivial, and thus $\Id\otimes Z_E:\R\otimes_{\cC_{CE}^*(\osp_{H(M)}^{<0})}\cC_{CE}^*(\osp_{H(M)}^{<0}\ltimes \GC_{H(M)})\cong \fGC_{H(M)}\ra \R$ is a map of cgdas that agrees with the partition function $Z_M$ in \eqref{CSI}.
		\item[(ii)] 
		It would be preferable to define a map $\cC_{CE}^*(\osp_{H(M)}^{<0};\TwGra_{H(M)})\ra \Omega_{dR}(\FM_E^{\fw})$ that is compatible with the comodule structure as in \cite{CW23}. However, as the maps $\pi_{V_{\text{ext}}(\Gamma)}:\FM_E^{\fw}[V(\Gamma)]\ra \FM_E^{\fw}[V_{\text{ext}}(\Gamma)]$ that forget the internal vertices are not submersions (unless there is at most one external vertex) we cannot define the fibre integral of differential forms. The standard way to deal with this problem is to use the more complicated model of piecewise semi-algebraic forms (see for example \cite{LV14,CW16}) and it seems very likely to the author that these methods can simply be applied here.
		
		But it is not strictly necessary for our purpose here. For one, Campos and Willwacher have shown that the algebraic model of $\FM_M$ only depends on the partition function $Z_M$, which can be defined either by de Rham- or piecewise semi-algebraic forms. It seems to be unknown whether these partition functions are equivalent but it is hard to believe that they are not. Moreover, the key property that we need is that the partition function is an invariant associated to $M$ with respect to either piecewise semi-algebraic or differential forms, and this is all we require really.
	\item[(iii)] Campos and Willwacher gave a combinatorial argument that the partition function $Z_M$ contains the same information as the naive real homotopy of $M$ (see Remark \ref{RemarkCamposWillwacher}(i)). This does not apply in the family case, i.e.\ the partition function $Z_E$ does not just depend on the underlying fibration if $\dim B>0$, even if $\dim M\geq 4 $ and $H^1(M)=0$. They also modified their algebraic models to describe configuration spaces of manifolds without a framing (essentially by removing tadpoles in their graphs). This argument can also not be extended for general fibre bundles without the framing, because the last part in the proof of \cite[Lem.\ 19]{CW23} does not work if $\dim B>0$.
	\end{itemize}
\end{rem}

\begin{proof}
	The Chevalley-Eilenberg complex $\cC_{CE}^*(\mathfrak{osp}_{H(M)}^{<0};\fGC_{H(M)})$ is isomorphic as a graded commutative algebra to $\cC_{CE}^*(\mathfrak{osp}_{H(M)}^{<0})\otimes \fGC_{H(M)}$ and the differential on $\cC_{CE}^*(\mathfrak{osp}_{H(M)}^{<0})$ is the usual Chevalley-Eilenberg differential. Hence, the restriction $Z_E|_{\cC_{CE}^*(\osp_{H(M)}^{<0})}$ is a map of cdgas by Proposition \ref{summary}. 
	
	\medskip
    The compatibility of $Z_E$ with the differential on $\fGC_{H(M)}$ reduces to a standard argument using Stokes' theorem
	\begin{equation}\label{dZE}
		dZ_E(\Gamma)=(-1)^{{|V(\Gamma)|\choose 2}d}\left(\int_{ \FM_E^{\fw}[V(\Gamma)]\ra B}d\omega(\gamma)-(-1)^{|\Gamma|+1} \int_{\partial^{\fw} \FM_E^{\fw}[V(\Gamma)]\ra B} \omega(\gamma)|_{\partial^{\fw}}\right).
	\end{equation}
	The differential on $\fGC_{H(M)}$ has three contributions 
	\[d=d_{\text{contr}}+d_{\text{split}}+d_{\text{action}},\]
	where the first two give the differential on $\fGC_{H(M)}$ and the third term encodes the action of $\osp_{H(M)}^{<0}$ on decorations of graphs in $\fGC_{H(M)}$. It follows from Proposition \ref{summary} and \eqref{pushout} that $d_{\text{action}}$ recovers the differential of $\Omega(E)$, and hence  that
	\begin{equation}\label{ZEdaction}
		Z_E(d_{\text{action}}(\Gamma))=(-1)^{{|V(\Gamma)|\choose 2}d}\int_{\FM_E^{\fw}[V(\Gamma)]\ra B}d\left(\bigwedge_{k=1}^D\pi_{i_k}^*\overline{\alpha}_k\right)\wedge\bigwedge_{l_1=1}^{|E_1(\Gamma)|}\pi^*_{m^1_{l_1}}\eta\wedge \bigwedge_{l_2=1}^{|E_2(\Gamma)|}\pi^*_{v^2_{l_2},w^2_{l_2}}\phi_{12},
	\end{equation}
	i.e.\,by replacing the decorations $\alpha_k$ by $D(\alpha_k)$ whose image in $\Omega_{dR}(E)$ is the exterior derivative $d\overline{\alpha}_k\in \Omega_{dR}(E)$. Here, we have used the expression for $\omega(\gamma)$ when $d$ is even and the analogous statement holds for $d$ odd.
	
	\smallskip
    Next, we deal with the splitting part of the differential. Denote by $\{x_i\}$ a basis of $H(M)$ and by $\{x_i^{\#}\}$ the dual basis with respect to the intersection pairing. Then by Proposition \ref{ospFC} and property (iii) of Proposition \ref{Prop}, the image of the diagonal class $\sum_i(-1)^{|x_i|}x_i\otimes x_i^{\#}\in H(M)\otimes H(M)$ in $\Omega_{dR}(E\times_B E)$ is $d\phi_{12}=\sum_i(-1)^{|x_i|}\pi_1^*\overline{x}_i\wedge \pi_2^*\overline{x}^{\#}_i$ and similarly $d\eta=\sum_{i}(-1)^{|x_i|}\overline{x}_i\wedge \overline{x}^{\#}_i$ by \eqref{PropFr}. Hence, it follows that
	\[Z_E(d_{\text{split}}(\Gamma))=(-1)^{{|V(\Gamma)|\choose 2}d}(-1)^{\sum_k|\alpha_k|}\int_{\FM_E^{\fw}[V(\Gamma)]\ra B}\bigwedge_{k=1}^D\pi_{i_k}^*\overline{\alpha}_k\wedge d\left(\bigwedge_{l_1=1}^{|E_1(\Gamma)|}\pi^*_{m^1_{l_1}}\eta\wedge \bigwedge_{l_2=1}^{|E_2(\Gamma)|}\pi^*_{m^2_{l_2},n^2_{l_2}}\phi_{12}\right)\]
	as $d_{\text{split}}$ replaces an edge with decorations by the diagonal class (see \eqref{dsplit}). Here, we have again given the formula for $d$ even and it has to be modified for $d$ odd. Combined with \eqref{ZEdaction}, we find that
	\begin{equation}\label{Intdomega}
	Z_E(d_{\text{split}}(\Gamma)+d_{\text{action}}(\Gamma))=\int_{ \FM_E^{\fw}[V(\Gamma)]\ra B}d\omega(\gamma).
	\end{equation}
	It remains to identify the fibre integral over the fibrewise boundary with $Z_E(d_{\text{contr}}(\Gamma))$. This argument is fairly standard (see for example \cite[App.\,A]{Wa09II} or \cite{CW23}) but for completeness we give a detailed discussion here.

	\smallskip
	It follows from Lemma \ref{FibrewiseModule} that the codimension 1 faces of the fibrewise boundary of $\FM_E^{\fw}[V]$ are the images of the module structure maps
	\[\circ_A:\FM_E^{\fw}[V/A]\times \FM_d[A]\lra \FM_E^{\fw}[V]\]
	for every $A\subset V$ with $|A|\geq 2$. We denote the restriction of $\omega(\gamma)$ to the stratum labelled by $A$ by $\omega(\gamma)|_{\partial_A}$, so that 
	\begin{equation}\label{BoundaryTerms}
		\int_{\partial \FM_E^{\fw}[V(\Gamma)]\ra B}\omega(\gamma)=\sum_{A\subset V(\Gamma),\,|A|\geq 2}\int_{\FM^{\fw}_E[V(\Gamma)/A]\times \FM_d[A]\ra B}\omega(\gamma)|_{\partial_A}
	\end{equation}
	We observe that the restriction $\omega(\gamma)|_{\partial_A}$ are pulled back from $\Omega_{dR}(\FM_E^{\fw}[V(\Gamma)/A])\otimes \Omega_{dR}(\FM_d[A])$: factors in \eqref{omegaGamma} coming from tadpoles and decorations are pulled back from $\Omega_{dR}(\FM_E^{\fw}[V(\Gamma)/A])$, and for edges between distinct vertices there are two different cases:
	\begin{itemize}
		\item[(I)] if $e\in E(\Gamma)$ connects vertices $v,w$ with at least one not in $A$, then by the commutativity of 
		\begin{equation*}
			\begin{tikzcd}
			\FM_E^{\fw}[V(\Gamma)/A]\times \FM_d[A]\arrow{r}{\circ_A} \arrow{d} & \FM_E^{\fw}[V(\Gamma)]\arrow{d}{\pi_{v,w}}\\
			\FM_E^{\fw}[V(\Gamma)/A]\arrow{r} & \FM_E^{\fw}[\{v,w\}]
			\end{tikzcd}
		\end{equation*}
		the restriction of $\pi^*_{v,w}\phi_{12}$ along $\circ_A$ is pulled back from  $\Omega_{dR}(\FM_E^{\fw}[V(\Gamma)/A])$;
		\item[(II)] if $e\in E(\Gamma)$ connects vertices $v,w\in A$, then by commutativity of 
		\begin{equation*}
		\begin{tikzcd}
		\FM_E^{\fw}[V(\Gamma)/A]\times \FM_d[A]\arrow{r}{\circ_A} \arrow{d} & \FM_E^{\fw}[V(\Gamma)]\arrow{d}{\pi_{v,w}}\\
		\FM_E^{\fw}[A/A]\times \FM_d[\{v,w\}]\arrow{r} & \FM_E^{\fw}[\{v,w\}]
		\end{tikzcd}
		\end{equation*}
		the restriction of $\pi_{v,w}^*\phi_{12}$ along $\circ_A$ agrees with the pull back of 
		\[\begin{cases} 
		   (-1)^d\text{vol}_{S^{d-1}}+\eta \in \Omega_{dR}(\FM_E^{\fw}[1])\otimes \Omega_{dR}(\FM_d[2]) & d \text{ even}\\
		   (-1)^d\text{vol}_{S^{d-1}} \in \Omega_{dR}(\FM_E^{\fw}[1])\otimes \Omega_{dR}(\FM_d[2]) & d \text{ odd}
		  \end{cases}
        \]
		along the left vertical map by Lemma \ref{Prop}(iv).
	\end{itemize}
	For a subgraph $\Gamma_A\subset \Gamma$ on $A$ that does not include decorations or tadpoles we define \[\text{vol}_{S^{d-1}}(\Gamma_A)=\bigwedge_{e\in \Gamma_A}\pi_e^*\text{vol}_{S^{d-1}}\in \Omega_{dR}(\FM_d[A]).\] 
	Then the restriction $\omega(\gamma)|_{\partial_A}$ is given by 
	\begin{align*}
	 \omega(\gamma)|_{\partial_A}=\begin{cases}
	                               \sum_{\Gamma_A\subset \Gamma}\pm\omega(\Gamma/\Gamma_A)\wedge \text{vol}_{S^{d-1}}(\Gamma_A)\in \Omega_{dR}(\FM_E^{\fw}[V(\Gamma)/A])\otimes \Omega_{dR}(\FM_d[A])& d \text{ even}\\
	                               \pm\omega(\Gamma/\Gamma_A^{\max})\wedge \text{vol}_{S^{d-1}}(\Gamma^{\max}_A)\in \Omega_{dR}(\FM_E^{\fw}[V(\Gamma)/A])\otimes \Omega_{dR}(\FM_d[A]) & d \text{ odd}
	                              \end{cases}
	\end{align*}
    where for even $d$ the sum is taken over all subgraphs $\Gamma_A\subset \Gamma$ on $A$ without decorations or tadpoles, and for odd $d$ only the maximal subgraph contributes. The sign is obtained by shuffling the factors into place.

    We first deal with the case that $|A|=2$. For even $d$ there are no multiple edges and the sum above contains at most two terms, depending on if there is an edge in $E(\Gamma)$ connecting the vertices in $A$. Clearly, the integral $\int_{\FM_d[A]}\text{vol}_{S^{d-1}}(\Gamma_A)$ can only be non-zero if there is an edge in $\Gamma_A$. For odd $d$, the form $\text{vol}_{S^{d-1}}(\Gamma_A)$ vanishes for degree reasons if there is more than one edge connecting the two vertices, and $\int_{\FM_d[A]}\text{vol}_{S^{d-1}}(\Gamma_A)$ vanishes if there is no edge in $\Gamma_A$ connecting the vertices in $A$.
 
	Hence, in both cases the non-trivial contribution in the fibre integral over the codimension 1 strata of the boundary with $|A|=2$ are given by the sum over the subsets $A=\{m^2_{l_2},n^2_{l_2}\}$ for $1\leq l_2\leq E_2$ corresponding to a single, non-loop edge connecting two vertices, i.e.
	\begin{align*}
		&\qquad \sum_{A\subset V(\Gamma),\,|A|=2}\int_{\FM_E^{\fw}[V(\Gamma)/A]\times \FM_d[A]\ra B}\omega(\gamma)|_{\partial_A} =\\
		&\sum_{l_2=1}^{|E_2(\Gamma)|}(-1)^{(d-1)(|E_2(\Gamma)|-l_2)}\int_{\FM_E^{\fw}[V(\Gamma)/\{m^2_{l_2},n^2_{l_2}\}]}\left(\omega^{D,1}(\gamma)\wedge \bigwedge_{l=1,\,l\neq l_2}^{|E_2(\Gamma)|}\pi^*_{[m^2_l],[n^2_l]}\phi_{12}\right) \int_{\FM_d[m^2_{l_2},n^2_{l_2}]}(-1)^d\text{vol}_{S^{d-1}}.
	\end{align*}
	Here, the sign is obtained by permuting the factor corresponding to $l_2$ to the end and $[m^2_l]$ and $[n^2_l]$ denote the points in the quotient $V(\Gamma)/\{m^2_{l_2},n^2_{l_2}\}$. One last subtlety to take into account is that the orientation of $\FM_M[V(\Gamma)/\{m^2_{l_2},n^2_{l_2}\}]\times \FM_d[m^2_{l_2},n^2_{l_2}]$ above is induced as part of the boundary of $\FM_d[V(\Gamma)]$. As the orientation of $\FM_d[m^2_{l_2},n^2_{l_2}]$ is fixed, we obtain an induced orientation of $\FM_M[V(\Gamma)/\{m^2_{l_2},n^2_{l_2}\}]$ which is potentially different than one induced via the identification $V(\Gamma)/\{m^2_{l_2},n^2_{l_2}\}\cong \{1,\hdots,|V(\Gamma)|-1\}$ used in the definition of $d_{\text{contr}}(\Gamma)$ in the operadic twisting construction. These two orientations agree up to a sign that is determined in Lemma \ref{OrientationBoundary}.

	All together, we find 
	\begin{align*}
	 &(-1)^{{|V(\Gamma)|\choose 2}d}(-1)^{|\Gamma|}\sum_{A\subset V(\Gamma),\,|A|=2}\int_{\FM_E^{\fw}[V(\Gamma)/A]\times \FM_d[A]\ra B}\omega(\gamma)|_{\partial_A}\\
	 =&(-1)^{{|V(\Gamma)|\choose 2}d} (-1)^{|\Gamma|}\sum_{l_2=1}^{|E_2(\Gamma)|}(-1)^{{|V(\Gamma)-1|\choose 2}d}(-1)^{(d-1)(|E_2(\Gamma)|-l_2)+d+d(m^2_{l_2}+n^2_{l_2}+1)}Z_E(\Gamma/(m^2_{l_2},n^2_{l_2}))\\
	 =&\sum_{l_2=1}^{|E_2(\Gamma)|}(-1)^{\sum|\alpha_i|+(d-1)(E_1+l_2)+d(m_{l_2}^2+n^2_{l_2}+1)}Z_E(\Gamma/(m^2_{l_2},n^2_{l_2}))
	\end{align*}
	where $\Gamma/(m^2_{l_2},n^2_{l_2})$ denotes the graph where we contract the edge labelled by $l_2$. This matches exactly the contraction part of the differential in \eqref{dcont} that we explained in detail in Appendix \ref{AlgebraicGraph}. Hence, we conclude 
	\[Z_E(d_{\text{contr}}(\Gamma))=(-1)^{{|V(\Gamma)|\choose 2}d}(-1)^{|\Gamma|}\sum_{A\subset V(\Gamma),\,|A|=2}\int_{\FM_E^{\fw}[V(\Gamma)/A]\times \FM_d[A]\ra B}\omega(\gamma)|_{\partial_A}.\]

\medskip
	The contribution of the boundary terms with $|A|>2$ vanishes by Kontsevich's vanishing lemma \cite[Lem.\,2.2]{Ko02}, i.e. for $d$ even
	\[ \sum_{A\subset V(\Gamma),\,|A|>2}\sum_{\Gamma_A\subset \Gamma}\pm\int_{\FM_E^{\fw}[V(\Gamma)/A]}\omega(\Gamma/\Gamma_A)\int_{\FM_d[A]}\text{vol}_{S^{d-1}}(\Gamma_A)=0\]
	and similarly for odd $d$ where only $\Gamma_A^{\max}$ contributes. In fact, Kontsevich showed that every summand vanishes individually and we have included a proof for completeness, which is identical for both even and odd dimensions.
	Let $A\subset V(\Gamma)$ with $|A|>2$, then there are several cases to consider:
	\begin{itemize}
		\item[(I)] Suppose that every vertex in $\Gamma_A$ is at least trivalent so that $2|E(\Gamma_A)|\geq 3|A|$. The integral $\int_{\FM_d[A]}\text{vol}_{S^{d-1}}(\Gamma_A)$ is non-zero only if the form is of top degree, i.e.\,if $(d-1)|E(\Gamma_A)|-d|A|=d|A|-d-1$. Combining these two gives conditions gives
		\[(d-3)|A|+2d+2\leq 0\]
		which is a contradiction for $d\geq 3$.
		\item[(II)] The graph $\Gamma_A$ has a bivalent vertex, say $a\in V(\Gamma_A)$ with adjacent vertices $b,c\in V(\Gamma_A)$, then
		\[\int_{\FM_d[A]}\text{vol}_{S^{d-1}}(\Gamma_A)=\int_{\FM_d[A\setminus a]}\text{vol}_{S^{d-1}}(\Gamma \setminus \{(ab),(ac)\})\int_{\FM_d[a,b,c]}\pi_{(ab)}^*\text{vol}_{S^{d-1}}\pi^*_{(ac)}\text{vol}_{S^{d-1}}.\] 
		There is an involution on $\FM_d[a,b,c]$ induced by the map that fixes $x_b$ and $x_c$ and sends $x_a$ to  $x_b+x_c-x_a$. A short calculation shows that the integral is anti-symmetric under this involution and hence vanishes, and therefore so does $\int_{\FM_d[A]}\text{vol}_{S^{d-1}}(\Gamma_A)$.
		\item[(III)] The graph $\Gamma_A$ has a vertex $a\in V(\Gamma_A)$ of valence $0$ or $1$. There can only be a contribution if $\text{vol}_{S^{d-1}}(\Gamma_A)$ is of top degree. But if $a$ is zero-valent then it is pulled back from $\FM_d[A\setminus a]$ and vanishes because $\dim\FM_d[A-a]=\dim \FM_d[A]-d$ if $|A|>2$. Similarly, if $a$ is univalent then \[\text{vol}_{S^{d-1}}(\Gamma_A)= \pi_{(a,b)}^*\text{vol}_{S^{d-1}}\wedge \text{vol}_{S^{d-1}}(\Gamma_A\setminus (a,b))\]
		where the second factor is pulled back from $\Omega_{dR}(\dim\FM_d[A-a])$. If $\text{vol}_{S^{d-1}}(\Gamma_A)$ is of top degree then the second factor has to vanish for degree reasons again.
	\end{itemize}
This concludes the proof that 
\begin{equation}\label{ZEdcontr}
-(-1)^{{|V(\Gamma)|\choose 2}d}(-1)^{|\Gamma|+1}\int_{\partial \FM_E^{\fw}[V(\Gamma)]\ra B}\omega(\gamma)=Z_E(d_{\text{contr}}(\Gamma)).
\end{equation}
Inputting \eqref{Intdomega} and \eqref{ZEdcontr} into \eqref{dZE}, we see that see that $Z_E$ is a chain map. 

\medskip

It remains to check that $Z_E$ is also compatible with the multiplication. For $i=1,2$ consider two graphs $\Gamma_i=[\gamma_i\otimes 1^{\otimes n_i}]\in \Tw \Gra_{H(M)}(\emptyset)$, where $n_i=|V(\Gamma_i)|$ denotes the number of (internal) vertices. By \eqref{TwMmultiplication} the multiplication is 
\[[(-1)^{n_1\cdot d\cdot |\gamma_2|}\pi^*_{\{1,\hdots,n_1\}}(\gamma_1)\cdot \pi^*_{\{n_1+1,\hdots,n_1+n_2\}}(\gamma_2)\otimes 1^{\otimes n_1+n_2}],\]
which is represented pictorially by the disjoint union $\Gamma_1\sqcup \Gamma_2$. Then
\begin{align*}
	Z_E(\Gamma_1\sqcup \Gamma_2)=(-1)^{{n_1+n_2\choose 2}d}(-1)^{n_1\cdot d\cdot |\Gamma_2|}\int_{\FM_E^{\fw}[V(\Gamma_1)\sqcup V(\Gamma_2)]}\omega(\gamma_1\sqcup \gamma_2).
\end{align*}
There is a map $\pr_1\times\pr_2:\FM_E^{\fw}[V(\Gamma_1)\sqcup V(\Gamma_2)]\ra \FM_E^{\fw}[V(\Gamma_1)]\times_B \FM_E^{\fw} [V(\Gamma_2)]$ and by definition the forms $\omega(\gamma_i)$ are pulled back from the respective factors. On the fibrewise interiors, this map is an embedding of an open dense subset and hence by naturality of fibre integration we have
\begin{align*}
Z_E(\Gamma_1\sqcup \Gamma_2)&=(-1)^{{n_1+n_2\choose 2}d}(-1)^{n_1\cdot d\cdot |\gamma_2|}\int_{\FM_E^{\fw}[V(\gamma_1)]\times_B \FM_E^{\fw}[ V(\gamma_2)]\ra B}\omega(\gamma_1)\wedge\omega(\gamma_2)\\
&\overset{\eqref{Fubini}}{=}(-1)^{{n_1+n_2\choose 2}d}(-1)^{n_1\cdot d\cdot |\gamma_2|}(-1)^{n_1d\cdot (|\gamma_2|-n_2d)}\int_{\FM_E^{\fw}[V(\Gamma_1)]\ra B}\omega(\gamma_1) \int_{\FM_E^{\fw}[V(\Gamma_2)]\ra B}\omega(\gamma_2)\\
&=(-1)^{{n_1\choose 2}d}(-1)^{{n_2\choose 2}d}\int_{\FM_E^{\fw}[V(\Gamma_1)]\ra B}\omega(\gamma_1) \int_{\FM_E^{\fw}[V(\Gamma_2)]\ra B}\omega(\gamma_2)\\
&=Z_E(\Gamma_1)Z_E(\Gamma_2).
\end{align*}
This concludes the proof that the partition function is a well defined map of cdga's. It remains to check that the homotopy class of $Z_E$ does not depend on the choice of model or propagator, which is the statement of the following lemma.
\end{proof}

\begin{lem}\label{WellDefined}
	The homotopy class of the partition function $Z_E:\cC_{CE}^*(\mathfrak{osp}^{<0}_{H(M)};\fGC_{H(M)})\ra \Omega_{dR}(B)$ does not depend on the choice of propagator or model.
\end{lem}
\begin{proof}
	Suppose we have two models that satisfy the conditions of Proposition \ref{summary}. By uniqueness of minimal semifree resolutions \cite[Thm 5.1]{ALF97}, any two such models are isomorphic
	and hence it suffices to prove that the partition functions $Z_E^i$ associated to two homotopic models $\phi_i:\Omega(E)\ra \Omega_{dR}(E)$ for $i=0,1$ are homotopic. Let $s\in \Hom^{-1}_{\Omega_{dR}(B)}(\Omega(E),\Omega_{dR}(E))$ be a homotopy between $\phi_0$ and $\phi_1$, i.e.\ $\phi_1-\phi_0=ds+sD$. Consider the bundle $\pi\times\Id:E\times I\ra B\times I$. Then 
	\[\Phi:=t \phi_1+(1-t)\phi_0 +dt\cdot s:\Omega_1(E)\ra \Omega_{dR}(E\times I)\] is a $\Omega_{dR}(B)$-linear quasi-isomorphism. By extending along $\pi_B^*:\Omega_{dR}(B)\ra \Omega_{dR}(B\times I)$, we obtain a model of $\pi\times \Id:E\times I\ra B\times I$ from Lemma \ref{Torelli}. The restriction to either end agrees with $\phi_i$ and in particular is compatible with the fibre integration pairing over $B\times \partial I$. Therefore, using this model in the proof of Lemma \ref{PairingTechnical}, we obtain a model of $E\times I$ compatible with fibre integration and which still agrees with $\phi_i$ over either end. Using $\Phi$ and a propagator on $\FM_{E\times I}^{\fw}[2]$ pulled back from $\FM_E^{\fw}[2]$, we obtain a partition function such that the restriction to either end of $B\times I$ agrees with $Z_E^i$, which are therefore homotopic.
	
	\smallskip
	Similarly, if we have two propagators $\phi_{12},\phi_{12}'\in \Omega_{dR}^{d-1}(\FM_E^{\fw}[2])$, then by Proposition \ref{PropUsed} they are cohomologous. Let $\beta\in \Omega_{dR}^{d-2}(E\times_BE)$ be a $(-1)^d$-symmetric form so that $\phi_{12}-\phi_{12}'=(\pi_1\times \pi_2)^*d\beta$. Then $t\phi_{12}+(1-t)\phi_{12}'+dt\wedge \beta$ is a propagator on $\FM_{E\times I}^{\fw}[2]=\FM_E^{\fw}[2]\times I$ with respect to the induced framing on $E\times I$ satisfying all the properties of Proposition \ref{PropUsed}, and the partition function associated to $\pi:E\times I\ra B\times I$ restricted to either end agrees with the partition function associated to $\phi_{12}$ and $\phi_{12}'$ respectively. Hence, they are homotopic.

\end{proof}

\section{Connection to automorphisms of configuration spaces}\label{Perspective}
In this last section, we discuss the conceptual origin of the construction of Kontsevich's classes, which has long been suspected to be related to the cohomology of the classifying space of automorphisms of the collection of configurations in a framed manifold considered as $\FM_d$-module. This was our motivation for trying to generalize Kontsevich's construction beyond framed homology disk bundles to general framed fibre bundles. The main result of this section provides further evidence about this close link.

\medskip
On the face of it, we study a framed fibre bundle $\pi:E\ra B$ through the collection of fibrewise configuration spaces $\FM_E^{\fw}:=\{\FM_E^{\fw}[n]\}_{n\geq 0}$ as a family of $\FM_d$-modules over $B$. Importantly, one can study operads and modules from a more homotopical point of view, e.g.\ by considering model categories of operads and modules in spaces where the weak equivalences are given by operad/module maps that are arity-wise weak equivalences. This could be extended to operads and modules in $\TOP/B$ to obtain a more homotopical notion of families of operads and modules over a fixed base space that would include $\FM_E^{\fw}$ above. To the author, it is reasonable to assume that there is a similar classification of families of right modules over the operad $B\times \FM_d$ in $\TOP/B$ as for fibrations of spaces in terms of homotopy automorphisms of the fibres, i.e.\ for $\FM_E^{\fw}$ we would expect to see the classifying space of the monoid of derived module homotopy automorphisms $\Aut^h_{\FM_d}(\FM^{\tau}_M)$, where $\tau\in \Iso(M\times \R^d,TM)$ is a fixed framing of $M$. If we consider the topological category $\mathscr{C}$ with space of objects given by $\Iso(M\times \R^d,TM)$ and morphisms spaces given by derived modules equivalences $\Map^h_{\FM_d}(\FM_M^{\tau},\FM_M^{\tau'})$, we can consider its classifying space $\B \mathscr{C}$, whose homotopy type is given by $\coprod_{[\tau]\in \pi_0(\mathscr{C})} \B\Aut^h_{\FM_d}(\FM_M^{\tau})$. Since any $f\in \Diff^+(M)$ induces a map $FM_M^{\tau}\ra \FM_M^{f_*\tau}$ of $\FM_d$-modules, there is a map 
\begin{equation*}
 \Iso(M\times \R^d,TM)\hcoker \Diff^+(M) \lra \coprod_{[\tau]\in \pi_0(\mathscr{C})} \B\Aut^h_{\FM_d}(\FM_M^{\tau}),
\end{equation*}
where the domain is the classifying space of framed fibre bundles with fibre $M$. Hence, one would expect that the invariants of a framed fibre bundle contained in the partition function $Z_E$ from Theorem \ref{PartFctZE} are pullbacks of cohomology class in $H^*(\B\Aut^h_{\FM_d}(\FM_M^{\tau});\R)$. 

\smallskip
The rational homotopy type of \emph{rationalized} configuration spaces $\FM_M^{\Q}$ of a framed manifold $M$ (which can be defined by extending the Sullivan realization from rational homotopy theory to operads and modules \cite{Wil24}) has been determined recently by Thomas Willwacher (see \cite[Cor.\ 13.8]{W23})\footnote{In fact, I was looking for a geometric construction of Kontsevich's classes for arbitrary manifolds when I learned of this result, which served as a starting point for this paper.} who proved that 
\[\B\Aut^h_{\FM_d^{\Q}}(\FM_M^{\Q})_{\Id}\simeq \left| \left|\MC_{\bullet}\left((\osp_{H(M)}^{<0}\ltimes \GC^{\geq 3}_{H(M)})^{z^{\geq 3}_M}\langle 0 \rangle \right)\right|\right|. \]
Here, $\GC_{H(M)}^{\geq 3}$ is the dg Lie algebra of connected graphs where all vertices are at least trivalent (see below) and $z_M^{\geq 3}\in \GC_{H(M)}^{\geq 3}$ is a Maurer-Cartan element that describes the rational homotopy type of $\FM_M$, and $\Aut^h_{\FM_d^{\Q}}(\FM_M^{\Q})_{\Id}$ denotes the path components of the derived automorphism spaces that induce the identity on $H^*(M;\Q)$ and which corresponds to our assumption that $\pi_1(B)$ acts trivially on $H(M)$. In other words, using the notation from \eqref{glM}, Willwacher showed that the dg Lie algebra $\gL_M$ (over $\Q$) is a dg Lie model for $\B\Aut^h_{\FM_d^{\Q}}(\FM_M^{\Q})_{\Id}$.

Our final result refines the information encoded in the partition function $Z_E$ to construct a map $\cC_{CE}^*(\gL_M)\lra \Omega_{dR}(B)$ of cgda's, which is exactly what we would expect if there was a classifying map $B\ra \B\Aut^h_{\FM_d}(\FM_M)_{\Id}$ of the family of $\FM_d$-modules associated to the framed fibre bundle with trivial fibre transport.
\begin{thm}\label{MainTheorem}
	Let $E\ra B$ be a framed, smooth submersion of connected spaces with closed fibre $M$ and $\dim M>2$ and suppose $\pi_1(B)$ acts trivially on $H(M)$. Then there exists $z_M^{\geq 3}\in \MC(\GC_{H(M)}^{\geq 3})$ that encodes the real homotopy type of the Fulton-MacPherson compactification of the fibres of $\pi$ as a $\FM_d$-model and a well defined map of cdgas
	\begin{equation}
	I:\cC_{CE}^*(((\osp^{<0}_{H(M)}\ltimes \GC^{\geq 3}_{H(M)})^{z^{\geq 3}_M})\langle 0 \rangle)\lra \Omega_{dR}(B)
	\end{equation}
	which is natural with respect to pullbacks and thus describes characteristic classes.
\end{thm}
We need to introduce two variations of the dg Lie algebra $\GC_{H(M)}$ introduced in \cite[Sect.\ 7]{CW23} that we use in the proof. First, they observe that the element
\[z_0=\sum_i (-1)^{|x_i|} \begin{tikzpicture}[vertex/.style={circle,fill=black,draw,minimum size=5pt,inner sep=0.2pt]},decoration/.style={}]
    \node[vertex] (1) at (0.3,0) {};
    \node[decoration] (2) at (0,0.6) {\mbox{\small $x_i$}};
    \node[decoration] (3) at (0.6,0.6) {\mbox{\small $x_i^{\#}$}};

    \draw[dotted] (3)-- (1) -- (2);
\end{tikzpicture}\in \GC_{H(M)}\]
is a Maurer-Cartan element, where $x_i\in H_*(M;\R)$ denotes a basis of the homology of $M$ but that the element $z_0$ is independent of that choice. Following the notation in \cite[Sect.\ 7.3]{CW23}, we denote by 
\[\GC_{H(M)}^{\geq 3}\subset \GC_{H(M)}''\subset \GC_{H(M)}^{z_0} \]
the sub dg Lie algebras of graphs that contain only three valent vertices, where the valence of a vertex is given by the number of incident edges plus the number of decorations, and respectively the sub dg Lie algebra of graphs that contain at least one trivalent vertex. It is checked in loc.\ cit.\ that these subspaces are indeed closed under the bracket and differential.
\begin{proof}
Since $\fGC_{H(M)}=\cC_{CE}(\GC_{H(M)})$ and
\[
 \cC_{CE}(\osp_{H(M)};\cC_{CE}(\GC_{H(M)}))\cong \cC_{CE}(\osp_{H(M)}\ltimes \GC_{H(M)}),
\]
the partition function in Theorem \ref{PartFctZE} corresponds to a Maurer-Cartan element 
\[z_E\in \MC(\Omega_{dR}(B)\hat{\otimes} (\osp_{H(M)}^{<0} \ltimes \GC_{H(M)}))\]
by \eqref{MCIdentification}. Then 
\[z_E':=z_E-z_0\in \MC((\Omega_{dR}(B)\hat{\otimes} (\osp_{H(M)}^{<0} \ltimes \GC_{H(M)}))^{1\otimes z_0}), \]
 and since $z_0$ vanishes under the $\osp_{H(M)}^{<0}$ action, we have 
\[
 (\Omega_{dR}(B)\otimes (\osp_{H(M)}\ltimes \GC_{H(M)}))^{z_0}=\Omega_{dR}(B)\otimes (\osp_{H(M)}\ltimes \GC_{H(M)}^{z_0}).
\]
By the same argument as in \cite[Lem.\ 45]{CW23}, which only relies on the properties of the propagator in Proposition \ref{PropUsed}, we observe that $z_E'$ lie in the subspace $ \Omega_{dR}(B)\hat{\otimes } (\osp_{H(M)}^{<0}\ltimes \GC_{H(M)}'')$. The inclusion $\GC_{H(M)}^{\geq 3} \ra \GC_{H(M)}''$ induces a weak equivalences of Maurer-Cartan spaces by \cite[Prop. 46]{CW23}, and the same argument applies to the dg Lie algebras obtained by tensoring with $\Omega_{dR}(B)$. Hence, $z_E'$ is equivalent to a Maurer-Cartan element coming from
\[z_E^{\geq 3}\in\MC(\Omega_{dR}(B)\hat{\otimes} (\osp_{H(M)}\ltimes \GC_{H(M)}^{\geq 3}))\]
which is unique up to contractible choice.

Finally, for a connected cdga $A$ and complete dg Lie algebra $L$ any Maurer-Cartan element $\tau\in (A\hat{\otimes}L)^1=\bigoplus_{i\leq 1} L_i\hat{\otimes}A^{1-i}$ can be decomposed as $\tau_1+\tau_2$ where $\tau_1\in L^1\hat{\otimes} A^0=L^1$ and $\tau_2\in \bigoplus _{i<1}L_i\hat{\otimes}A^{1-i}$, and $\tau_2$ is a Maurer-Cartan element in $(L\hat{\otimes}A)^{\tau_1}$. The Maurer-Cartan equation implies that $\tau_2$ is in fact a Maurer-Cartan element in the truncated dg Lie algebra $L^{\tau_1}\langle 0 \rangle \hat{\otimes}A$, and this determines a map of cdga's 
\[\cC_{CE}(L^{\tau_1}\langle 0\rangle ) \xrightarrow{\ev(\tau_2)} A\]
by evaluating $\tau_2\in \MC(L^{\tau_1}\langle 0\rangle \hat{\otimes}A)$. We can replace the de Rham complex by a connected Sullivan model $\phi:\Lambda\xrightarrow{\simeq}\Omega_{dR}(B)$ since $B$ is connected by assumption. This induces a weak-equivalence of dg Lie algebras
\[
\Phi:\Lambda\hat{\otimes} (\osp_{H(M)}\ltimes \GC_{H(M)}^{\geq 3}))\lra \Omega_{dR}(B)\hat{\otimes} (\osp_{H(M)}\ltimes \GC_{H(M)}^{\geq 3}))
\]
compatible with the complete filtration induced from $\osp_{H(M)}\ltimes \GC_{H(M)}^{\geq 3}$. By another application of the Goldman-Milson theorem \cite[Thm 1.1]{DR15}, $z_E^{\geq 3}$ is equivalent to a the image of a Maurer-Cartan element 
\[ z^{\geq 3}_{\Lambda}\in \Lambda\hat{\otimes} (\osp_{H(M)}\ltimes \GC_{H(M)}^{\geq 3})\]
under $\Phi$ which is unique up to contractible choice. Decomposing $z_{\Lambda}^{\geq 3}= z^{\geq 3}_M+z_{\Lambda}'$, where $z_M\in \osp_{H(M)}\ltimes \GC_{H(M)}^{\geq 3}$ and $z'_{\Lambda}$ contains the summands with elements in $\Lambda $ of positive degree, we obtain a map 
\[\cC_{CE}((\osp_{H(M)}\ltimes \GC_{H(M)}^{\geq 3})^{z^{\geq 3}_M}\langle 0 \rangle ) \lra \Lambda \xrightarrow{\simeq} \Omega_{dR}(B)\]
by the above argument. Finally, observe that $z^{\geq 3}_M\in \GC_{H(M)}^{\geq 3}$ for degree reasons, and that $z_0+z_M\in \GC_{H(M)}$ is a Maurer-Cartan element that is gauge equivalent by construction to the restriction of $z_E$ to some point $b\in B$ via the induced augmentation $\Omega_{dR}(B)\ra \R$. Hence, it encodes the real homotopy type of of the Fulton-MacPherson compactification of the fibre $\pi^{-1}(b)$ as a $\FM_d$-module by \cite[Thm 25]{CW23} for any $b\in B$.

Finally, naturality under pullback follows directly from the naturality of the fibrewise partition function $Z_E$.
\end{proof}

The dg Lie algebra $(\osp_{H(M)}^{<0}\ltimes \GC_{H(M)})^{z^{\geq 3}_M}\langle 0 \rangle$ appears in Willwacher's work because it acts by biderivations on the algebraic model for $\FM_M$. So Willwacher's result is analogous to an old result of Sullivan that shows that the dg Lie algebra of derivations of a (Sullivan) model of a space $X$ models the classifying space $\B\haut_0(X)$ of the identity component of the space. Moreover, a relative (Sullivan) model of a fibration $E\ra B$ with fibre $X$ determines a model for the classifying map. We expect that Theorem \ref{PartFctZE} can be extended to provide a relative model for $\FM_E^{\fw}$ as dg Hopf comodule over the cooperad $\Omega_{dR}(B)\otimes \mathrm{Graphs}_d$ in the category of $\Omega_{dR}(B)$-algebras (see Remark \ref{RemarkFamilyCSI}(ii)), which by analogy with relative models of fibrations would be the same as giving a model of the classifying map $B\ra\B\Aut^h_{\FM_d}(\FM_M)_{\Id}$. From that point of view, the configuration space integral appears because it is a convenient tool that encodes the real homotopy type of configurations spaces, which is quite special to modules of configuration space type \cite[Sect.\ 4]{W23}.

\begin{rem}
It has long been suspected that Kontsevich's characteristic classes are related to embedding calculus. From a modern point of view, the approximation from embedding calculus can be defined as the derived mapping space 
 	\[T_{\infty}\Emb(M,N)=\Map^h_{\FM_d^{\fr}}(\FM_M^{\fr},\FM_N^{d-\fr})\] 
of framed configuration spaces as right modules over the framed Fulton-MacPherson operad (see \cite{Tu13,BrW13}), where $M$ is a $d$-dimensional manifold and $\FM_N^{d-\fr}$ denotes (compactified) configurations spaces of $N$ where each particle $p$ is decorated by $d$ linear independent vectors in $T_pN$.

Theorem \ref{MainTheorem} indeed highlights that the configuration space integral provides a geometric construction for framed fibre bundles of the pullback of (real) cohomology classes of the classifying space associated to framed self-embedding calculus, i.e.\ derived mapping spaces of configuration spaces of framed manifolds considered as $\FM_d$-modules.
\end{rem}

\appendix

\section{Cooperadic twisting}
The motivation for recalling the operadic twisting construction here is to give a precise definition of the graph complex that is as short and concise as possible and at the same time makes explicit all choices involved so that the diligent reader can check the computation for themselves.

\medskip
We need to recall a few preliminaries about operads for this purpose. In the following, we work with operads and modules in the symmetric monoidal category of cochain complexes over a field $k$ of characteristic zero, using the terminology from \cite{LV12}.\footnote{Observe, however, that in \cite{LV12} they study operads in the category of chain complexes whereas we work with cochain complexes, and the isomorphism of categories interchanges between suspension and desuspension.} For a cochain complex $V$ we denote by $\End_V$ the \emph{endomorphism operad} given by $\{\Hom(V^{\otimes n},V)\}_{n \geq 0}$ and by $\coEnd_V$ the \emph{co-endomorphism operad} $\{\Hom(V,V^{\otimes n})\}_{n\geq 0}$. The \emph{$r$-th suspension of an operad} $P$ is defined as $P\{r\}:=P\otimes_H\End_{k[-r]}$, where $\otimes_H$ is the Hadamard product and $k[-r]$ is the one dimensional cochain complex concentrated in degree $r$. Similarly, the \emph{$r$-th suspension of a cooperad} $C$ is defined as $C\{r\}:=C\otimes_H(\coEnd_{k[-r]})^*$. Using the canonical isomorphism $\Hom(V,W)^*\cong \Hom(W,V)$ for finite dimensional vector spaces, the $N$-ary (co)operations of the suspended (co)operads are given by
	\begin{align*}
	P\{r\}(N)&=P(N)\otimes \Hom(k[-r]^{\otimes N},k[-r]),\\
	C\{r\}(N)&=C(N)\otimes \Hom(k[-r]^{\otimes N},k[-r]).
	\end{align*}
We denote by $\Lie$ the Lie operad and by $\Com$ the (non-unital) commutative operad, then the Koszul dual cooperad of $\Lie$ is given by $\Lie^{\vee}=\Com^*\{1\}$. Following \cite{Wil15} we denote the minimal resolution of $\Lie\{d-1\}$ by
\[\hoLie_d:=\Omega(\Com^*\{d\}),\]
where $\Omega$ denotes the cobar construction of a coaugmented cooperad. The operadic twisting construction of Willwacher in \cite[App.\ I]{Wil15} associates to an operad $P$ with a map $\mu:\hoLie_d\ra P$ a new operad $\text{Tw }P$ that governs $P$ algebras with a differential twisted by $\mu$. Moreover, given a right $P$-module $M$, there is an associated right $\text{Tw }P$-module $\text{Tw }M$. If the operad is the linear dual of a cooperad $C$, there are analogous predual constructions of the twisted cooperads and comodules. The graph complexes relevant in this work are all obtained by this cooperadic twisting procedure applied to the $\Gra_d$-comodule $\Gra_{H(M)}$.

\medskip
 A map of operads $\mu:\hoLie_d\ra C^*$ is the same as a Maurer-Cartan element in the convolution Lie algebra $\Hom_{\mathbb{S}}(\Com^*\{d\},C^*)$ which can be identified by dualizing with
\[\Hom_{\mathbb{S}}(C,\Com\{-d\})=\prod_{N\geq 0}\Hom_{S_N}(C(N),\Com\{-d\}(N))\]
and that we denote by $\gL$ below. By \cite[6.4.4]{LV12}, the pre-Lie algebra structure $f*g$ of two elements is given by the composition
\[C\xrightarrow{\,\De_{(1)}\,}C\circ_{(1)}C\xrightarrow{f\circ_{(1)}g} \Com^*\{-d\}\circ_{(1)}\Com^*\{-d\} \xrightarrow{\,\gamma_{(1)}} \Com^*\{-d\}\]

Let $\mu:\hoLie_d\ra C^*$ be as above and $M$ be a right $C$-comodule. There is a cooperad $\Tw C$ and a right $\Tw C$-comodule $\Tw M$ whose space of operations are defined by
\begin{align*}
\Tw C(U)&=\bigoplus_{j\geq 0} \left(C(U\sqcup \underline{j})\otimes k[d]^{\otimes j}\right)_{S_j},\\
\text{Tw }M(U)&=\bigoplus_{j\geq 0} \left(M(U\sqcup \underline{j})\otimes k[d]^{\otimes j}\right)_{S_j},
\end{align*}
with an internal differential $d_{\text{internal}}$ coming from the differential on $C$ or $M$ respectively. The multiplication (on $\Tw M$) is defined as follows 
\begin{equation}\label{TwMmultiplication}
 [m_1\otimes 1^{\otimes n_1}]\cdot [m_2\otimes 1^{\otimes n_2}]:=(-1)^{|m_2|\cdot n_1\cdot d}[\pi^*_{\{1,\dots,n_1\}} m_1 \cdot \pi^*_{\{n_1+1,\dots,n_1+n_2\}}(m_2) \otimes 1^{\otimes n_1+n_2}],
\end{equation}
where $\pi^*_{\{1,\dots,n_1\}}:M(n_1)\ra M(n_2)$ denotes the $C(M)$-comodule structure map corresponding to $\{1,\dots,n_1\}\subset \{1,\dots,n_1+n_2\}$ and similarly for $\pi^*_{\{n_1+1,\dots,n_1+n_2\}}$.

Below, we only discuss details for the differential of the twisting construction of $\Tw M$ and we refer to \cite[App.\ I]{Wil15} for more details about the (dual of the) cooperad and comodule structure. The twisted differential arises via the action of the convolution Lie algebra $\gL$ on $\Tw M$, that we define for homogeneous elements 
\begin{align*}
	x&\in (C^*(k)\otimes \mathrm{sgn}_{S_k}^{\otimes d}[-d(k-1)])^{S_k}\cong \Hom_{S_k}(C(k),\Com\{-d\}(k))\subset \gL,\\
	[m\otimes 1^n]&\in (M(n)\otimes k[d]^{\otimes n} )_{S_n}\subset\Tw M(\emptyset),
\end{align*}
as follows 
\begin{equation}\label{Defaction}
	x\cdot [m\otimes 1^n]=\sum_{I\subset \underline{n},\,|I|=k} \sgn(I)^d\cdot \left[(1\otimes x)(\circ_I(m))\otimes 1^{n-k+1}\right].
\end{equation}
Here, we identify $I\cong \underline{k}$ in the unique order preserving way and $\underline{n}/I\cong \underline{n-|I|+1}$ so that $[I]$ is mapped to $1$ and order preserving on $\underline{n}\setminus I$, so that we can evaluate $1\otimes x$ on $M(\underline{n}/I)\otimes C(I)\cong M(n-k+1)\otimes C(k)$. Further, $\sgn(I)$ denotes the sign of the $(k,n-k)$-shuffle determined by $(I,\underline{n}\setminus I)$. This is the dual of the formula in \cite[Lem. I.4]{Wil15} and Willwacher has shown that it defines an action, but we give a short proof here as well.
\begin{lem}
Equation \eqref{Defaction} defines a left action of the convolution Lie algebra $\gL$ on $\Tw M$.
\end{lem}
\begin{proof}
	The action can be defined as a composition of maps of $\Ss$-modules as follows. Denote by $k[d]$ the $\Ss$-module with $k[d](j)=k[d]^{\otimes j}$ and by $x: C\ra \Com\{d\}$ the corresponding map of $\Ss$-modules $C\ra \Com\{d\}$, then we consider the map 
	\begin{equation}\label{xInfinitesimal}
	M\otimes k[d]\xrightarrow{\De_{(1)}\otimes \Id} (M\circ_{(1)} C)\otimes k[d]\xrightarrow{(\Id\circ_1 x)\otimes \Id} (M\circ _{(1)}\Com\{d\})\otimes k[d],
	\end{equation}
	where $\circ_1$ denotes the infinitesimal composition product \cite[Sect.\ 6.1]{LV12}. Observe that 
	\[(M\circ _{(1)}\Com\{d\})\otimes k[d](n)=\left(\bigoplus_{\emptyset\neq I\subset n}M(\underline{n}/I)\otimes \Hom(k[d]^{\otimes I},k[d])\right)\otimes k[d]^{\otimes n}\]
	and there is an evaluation map $ \Hom(k[d]^{\otimes I},k[d])\otimes k[d]^{\otimes n}\ra k[d]^{\otimes \underline{n}/I}$ which contributes the sign $\sgn(I)^d$ if we identify $\underline{n}/I\cong \underline{n-k+1}$ as in \eqref{Defaction}. Identifying $\bigoplus_{\emptyset \neq I\subset \underline{n}}M(n/I)\otimes k[d]^{\otimes \underline{n}/I}$ with $((M\otimes k[d])\circ_{(1)}\Com)(n)$, the evaluation determines a map of $\Ss$-modules \[\ev:(M\circ _{(1)}\Com\{d\})\otimes k[d]\lra (M\otimes k[d])\circ_{(1)} \Com,\] and the composition with \eqref{xInfinitesimal} gives
	\[M\otimes k[d]\ra (M\otimes k[d])\circ_{(1)}\Com.\]
	Since for any $\Ss$-module $N$ we have $(N\circ_{(1)}\Com)(n)_{S_n}\cong \bigoplus_{k=1}^{n}N(k)_{S_{k-1}}$, there is a natural map $\bigoplus_k(N\circ_{(1)}\Com)(k)_{S_k}\ra \bigoplus_kN(k)_{S_k}$. The action in \eqref{Defaction} is given by the induced map of the composition $M\otimes k[d]\ra (M\otimes k[d])\circ_{(1)}\Com$ above on the sum over all coinvariant spaces.
	
	\smallskip
	It remains to show that this is a Lie algebra action. For (homogeneous) elements $x,y\in \gL$ the action $y\cdot (x\cdot m)$ is the same as the induced map of the composition
	\begin{equation*}
		\begin{tikzcd}
		M\otimes k[d] \arrow{r}{\De_{(1)}} & (M\circ_{(1)}C)\otimes k[d] \arrow{r}{\De_{(1)}}& ((M\circ_{(1)}C)\circ_{(1)} C)\otimes k[d] \arrow{d}{(\Id\circ_{(1)} y)\circ_{(1)} x}\\
		&& ((M\circ_{(1)}\Com\{d\})\circ_{(1)} \Com\{d\})\otimes k[d]\arrow{d}{\ev}\\
		&& ((M\circ_{(1)} \Com\{d\})\otimes k[d])\circ_{(1)} \Com\arrow{d}{\ev\circ_{(1)}\Id}\\
		&& ((M\otimes k[d])\circ_{(1)} \Com)\circ_{(1)} \Com
		\end{tikzcd}
	\end{equation*}
	The infinitesimal composition product is not associative and $(M\circ_{(1)} C)\circ_{(1)} C$ naturally splits into two direct summands, one of which is given by $M\circ_{(1)}(C\circ_{(1)}C)$ where this composition agrees with the action of the pre-Lie product $y*x$. A straight-forward check shows that the contributions on the second summand cancel with that of $x\cdot (y\cdot m)$ so that the action \eqref{Defaction} is compatible with the Lie algebra structure on the convolution dg Lie algebra $\gL$. 
\end{proof}

Observe that for any cooperad $C$ and any comodule $M$, the algebra $C^*(1)$ acts from the left on $M$ (as well as $C$) by comodule (cooperadic) derivations via $x \cdot m:=\sum_{i=1}^n (1\otimes x)\circ_i(m)$, where $m\in M(n)$ and $x\in C^*(1)$ (and a similar formula for $C$). Hence, \[(\Tw C(1))^*\cong \prod_{j\geq 0}(C^*(1\sqcup \underline{j})\otimes k[-d]^{\otimes j})^{S_j}\]
acts from the left on $\Tw M$. Moreover, since $\gL$ also acts by cooperadic derivations, it follows that $\gL\ltimes (\Tw C(1))^*$ acts on $\Tw M$. Identifying 
\[\gL\cong \prod_{j\geq 0} (C^*(j)\otimes k[-d]^{\otimes j})^{S_j}[d] \subset \prod (C^*(j)\otimes k[-d]^{\otimes j})^{S_{j-1}}[d]\cong (\Tw C(1))^*,  \]
we can associate to any element $x\in \gL$ an element $x_1\in (\Tw C(1))^*$ and Willwacher has shown \cite[Lem.\ I.5]{Wil15} that 
\[\hat{\mu}:=\mu-\mu_1 \in \MC(\gL\ltimes \Tw C(1)^*).\]
The differential on the twisting construction $\Tw C$ and $\Tw M$ can now be defined as 
\[d_{\Tw M}=d_{\text{internal}}+\hat{\mu}\cdot -.\]
Observe that the formula for the differential simplifies in arity zero because $(\Tw C(1))^*$ acts trivially so that the differential simplifies to $d_{\Tw M}=d_{\text{internal}}+\mu\cdot -$. In the next section, we compute the differential for $\Tw \Gra_{H(M)}(\emptyset)$.

\subsection{The (full) graph complex via operadic twisting}\label{AlgebraicGraph}
In this brief section we work out the twisting construction for the right $\Gra_d$-comodule $\Gra_{H(M)}$ from Definition \ref{GraV}. The twisting map $\mu:\hoLie_d\ra \Gra_d^*$ is determined by the map of $\Ss$-modules $\mu:\Gra_d\ra \Com\{d\}$ whose only non-trivial contribution is in arity $2$ given by $\mu(s^{12})=(-1)^{d+1}$ for $s^{12}\in \Gra_{d}(2)$ and which vanishes for all other elements.

\medskip
Consider a homogeneous element in $\Gamma=[\gamma\otimes 1^n]\in \fGC_{H(M)}=\Tw \Gra_{H(M)}(\emptyset)$ where $\Gamma$ is represented pictorially as a graph on $\underline{n}$ with decorations $\alpha_k$ at vertices $i_k$ and an edge between between $m_l$ an $n_l$ for every factor in the second contribution. The orientation of $\Gamma$ determined by the order of the factors and the labelling of the vertices by $\underline{n}$ (this orientation depends on the parity of $d$). The $D$ and $E$ above stand for the number of decorations and edges of $\Gamma$ respectively.

Furthermore, we can decompose the set of edges into $E(\Gamma)=E_1(\Gamma)\sqcup E_2(\Gamma)$ where $E_1(\Gamma)$ contains all tadpoles if $d$ is even (i.e.\,all factors corresponding to $l$ with $m_l=n_l$) or multiple edges if $d$ is odd (i.e.\,all factors corresponding to $l$ for which there $l'\neq l$ with $m_l=m_{l'}$ and $n_l=n_{l'}$), and $E_2(\Gamma)$ contains the other edges, i.e\ $\gamma$ is represented by
 \begin{equation*}
	\gamma:=\prod_{k=1}^{D}\pi_{i_k}^*\alpha_k\prod_{l_1=1}^{E_1}s^{m^1_{l_1},n^1_{l_1}} \prod_{l_2=1}^{E_2}s^{m^2_{l_2},n^2_{l_2}}  \in \Gra_{H(M)}(n).
\end{equation*}
Observe that $\mu$ acts non-trivially only on edges in $E_2(\Gamma)$. Then, with the evident adaption in notation, we find  
\begin{equation}\label{dcont}
	\mu\cdot \Gamma=(-1)^{\sum_k|\alpha_k|+(d-1)E_1}\sum_{l_2=1}^{E_2}(-1)^{d(m_{l_2}+n_{l_2}+1)}(-1)^{(d-1)l_2}\left[\prod_{k=1}^D\pi^*_{\bar{i}_k}\alpha_k\prod_{l_1=1}^{E_1}s^{\bar{m}^1_{l_1},\bar{n}^1_{l_1}}\prod_{l'_2=1,l'_2\neq l_2}^{E_2}s^{\bar{m}^2_{l'_2},\bar{n}^2_{l'_2}}\otimes 1^{n-1}\right].
\end{equation}
Hence, $\mu\cdot \Gamma$ is represented pictorially by the sum over all edge contractions and \eqref{dcont} gives a precise statement thereof.

\section{Orientation conventions and fibre integration}\label{FibreIntegration}
We have to fix various orientation conventions in order to apply Stokes' theorem consistently. Denote the orientation of a manifold $M^d$ by a nowhere vanishing section $o(M)\in \Gamma(\Lambda^d T^*M)$, then the induced orientation on the boundary is defined as $i_Xo(M)|_{\partial M}$, where $X$ is an outward pointing vector field on the boundary $\partial M$ and $i_X$ denotes the interior product. With this convention, Stokes' theorem is $\int_Md\omega=\int_{\partial M}\omega|_{\partial M}$ for a (compactly supported) differential form $\omega\in \Omega_{dR}^d(M)$. 

The orientation of the total space $E$ of an oriented, smooth submersion $\pi:E\ra B$ over an oriented manifold $B$ and fibre $M$ is chosen as $o(E)=o(B)\wedge o(M)$, where $o(M)$ is a nowhere vanishing section in $\Gamma(\Lambda^d T^*_{\pi}E)$. If the fibre is compact, then we can define the fibre integration
\begin{equation}
\begin{split}
	\int_{\pi}:\Omega^*_{dR}(E)\lra \Omega^{*-d}_{dR}(B)
\end{split}
\end{equation}
by the following evaluation of the fibre integral of $\alpha\in \Omega_{dR}^{p+d}(E)$ on vectors $X_1,\hdots,X_p\in T_bB$
\[\left(\int_{\pi}\alpha\right)_b(X_1,\hdots,X_p)=\int_{\pi^{-1}(b)}i_{\overline{X}_p}\cdots i_{\overline{X}_1}\alpha,\] 
where $\overline{X}_i$ are smooth sections of $TE|_{\pi^{-1}(b)}$ that lift $X_i$. Such sections can easily be constructed and the definition of fibre integration is independent of the choice.

\medskip
Then with the conventions above, an elementary check shows that 
\begin{equation}
	\begin{split}
	\int_E\alpha&=\int_B\int_{\pi}\alpha\\
	\int_{\pi}\pi^*(\beta)\wedge \alpha&=\beta\wedge \int_{\pi}\alpha
	\end{split}
\end{equation}
for $\alpha\in \Omega_{dR}(E)$ and $\beta\in \Omega_{dR}(B)$. Moreover, the fibrewise versions of Stokes' and Fubini's theorem with the above definition are as follows.
\begin{lem}
	Let $\pi:E\ra B$ be an oriented, smooth submersion with fibre $M^d$ and oriented base $B$. If $M$ is a manifold with boundary, the fibrewise boundary $\pi^{\partial}:\partial^{\fw}E\ra B$ is a smooth, oriented submersion with fibre $\partial M$ and 
	\begin{equation}
		d\int_{\pi}\alpha=\int_{\pi}d\alpha -(-1)^{|\alpha|-d+1}\int_{\pi^{\partial}}\alpha|_{\partial^{\fw}E}.
	\end{equation}
\end{lem}
\begin{lem}\label{FibrewiseFubini}
 Let $\pi_i:E_i\ra B$ for $i=1,2$ be two oriented submersions with compact fibres of dimension $d_i$. Denote by $\pr_i:E_1\times _BE_2\ra E_i$ the projections, then for $\alpha_i\in \Omega_{dR}^{a_i}(E_i)$ the following identity holds
 \begin{equation}\label{Fubini}
    \int_{\pi_1\times\pi_2:E_1\times_BE_2\ra B} \pr_1^*\alpha_1\wedge \pr_2^*\alpha_2=(-1)^{d_1(a_2-d_2)}\int_{\pi_1}\alpha_1 \, \wedge \, \int_{\pi_2}\alpha_2.
 \end{equation}
\end{lem}
\subsection{Orientation of the boundary faces of \texorpdfstring{$\FM_M$}{FMM}}

As a last step, we need to discuss the orientation of the Fulton-MacPherson compactification or rather the induced orientation on its boundary faces. Namely if we orient $\FM_M[n]$ by $\pi_1^*o(M)\wedge\hdots\wedge \pi_n^*o(M)$, then there is an induced orientation of $\FM_d[\underline{n}/A]$ via the inclusion of (a piece of) the boundary $\circ_A:\FM_M[\underline{n}/A]\times \FM_d[A]\ra \FM_M[\underline{n}]$ and the fixed orientation on $\FM_d[A]$. In the twisting construction, the we identify $\underline{n}/A \cong \underline{n-|A|+1}$ so that $[A]$ corresponds to $1$ and order preserving on $\underline{n}\setminus A$, which induces a potentially different orientation of $\FM_M[\underline{n}/ A]$ than the one above. The following lemma compares these two orientations of $\FM_M[\underline{n}/A]$ in the case $|A|=2$, which is the only relevant case for Proposition \ref{PartFctZE}. 
\begin{lem}\label{OrientationBoundary}
	Let $A=\{i,j\}$ for $1\leq i<j\leq n$. Then the orientation of $\FM_d[\underline{n}/A]$ induced by the boundary inclusion $\circ_A:\FM_d[\underline{n}/A]\times\FM_d[A]\ra \FM_d[\underline{n}]$ agrees with the standard orientation of $\FM_d[\underline{n-1}]$ under the above identification up to the sign $(-1)^{d(i+j+1)}$.
\end{lem}
\begin{proof}
	The orientation of $\FM_d[2]$ is chosen so that the orientation of $\FM_M[1]\times \FM_d[2]$ agrees with the induced orientation as the boundary of $\FM_M[2]$. Hence, the two orientations on $\FM_M[\underline{n}/A]$ agree for $A=\{1,2\}$, which is similar to the argument that for manifolds $N_1,N_2$ with $\partial N_2=\emptyset$, the identity $\partial (N_1\times N_2)=\partial N_1\times N_2$ is orientation preserving with respect to the induced orientation of the boundary on the left side and the product orientation on the right side. For general $A$, this can be compared to this case via the shuffle permutation corresponding to $(\{i,j\},\underline{n}\setminus \{i,j\})$ that contributes a sign $(-1)^{d(i+j-1)}$.
\end{proof}

\bibliographystyle{alpha}
\bibliography{KontsevichClasses.bbl}

\bigskip
\textit{Email address}: nils.prigge@gmail.com

\end{document}